\documentclass[UTF8]{amsart}
\usepackage{amsmath,amssymb,graphicx,bbm,amsthm}
\usepackage{caption}
\usepackage{subcaption}
\usepackage[american]{babel}
\usepackage{color}
\usepackage{hyperref}
\hypersetup{colorlinks = true, allcolors = blue}
\usepackage[all]{xy}
\usepackage{tikz-cd}
\newtheorem{theorem}{Theorem}[section]
\newtheorem{lemma}[theorem]{Lemma}
\newtheorem{proposition}[theorem]{Proposition}
\newtheorem{claim}[theorem]{Claim}

\newtheorem{corollary}[theorem]{Corollary}

\numberwithin{equation}{section}

\theoremstyle{definition}

\theoremstyle{example}

\theoremstyle{remark}
\newtheorem{remark}[theorem]{Remark}

\DeclareMathOperator{\sys}{sys}

\DeclareMathOperator{\arccosh}{arccosh}

\DeclareMathOperator{\pmodo}{PMod}
\DeclareMathOperator{\modo}{Mod}
\usepackage{xcolor}
\usepackage{verbatim}

\usepackage{listings}
\usepackage{color} 
\definecolor{mygreen}{RGB}{28,172,0} 
\definecolor{mylilas}{RGB}{170,55,241}

\begin{document}

\lstset{language=Matlab,
    breaklines=true,
    morekeywords={matlab2tikz},
    keywordstyle=\color{blue},
    morekeywords=[2]{1}, keywordstyle=[2]{\color{black}},
    identifierstyle=\color{blue},
    stringstyle=\color{mylilas},
    commentstyle=\color{mygreen},
    showstringspaces=false,
    numbers=left,
    numberstyle={\tiny \color{black}},
    numbersep=9pt, 
	xleftmargin=1cm
}

\title{Thurston spine in a Teichm\"uller curve}
\author{Yue Gao}
\address{School of Mathematics and Statistics, Anhui Normal University, Wuhu, Anhui, China}
\email{yuegao@ahnu.edu.cn}

\author{Zhongzi Wang}
\address{School of Mathematical Sciences, Peking University, Beijing, China}
\email{wangzz22@stu.pku.edu.cn}

\date{}

\maketitle

\begin{abstract}
 To study the Thurston spine $\mathcal{P}_g \subseteq \mathcal{T}_g$, we construct a Teichm\"uller curve $V \subseteq \mathcal{T}_g$. Then we characterize $V \cap \mathcal{P}_g$. More specifically, we show it is a trivalent tree and is an equivariant deformation retract of $V$. 
 Moreover, by our construction, a lot of essential loops in the Thurston spine, both reducible and pseudo-Anosov, are obtained. 
\end{abstract}

\tableofcontents

\addtocontents{toc}{\protect\setcounter{tocdepth}{1}}

\section{Introduction}

On a hyperbolic surface $X$, the shortest closed geodesic is called the \emph{systole} of $X$. 
Systole plays a role not only in the geometry of $X$, but also in the geometry and topology of the spaces of all the genus $g$ hyperbolic surfaces, namely the moduli space of hyperbolic surfaces $\mathcal{M}_g$ and the Teichm\"uller space $\mathcal{T}_g$, see \emph{e.g.} \cite{mumford1971remark, schaller1999systoles, akrout2003singularites, bourque2021local, wu2024systole}. 
In particular, surfaces with filling systole are of specific interest. A set of geodesics in $X$ is \emph{filling} if the geodesics cut the surface into disks. In an unpublished manuscript \cite{thurston1986spine}, Thurston conjectured that in the moduli space $\mathcal{M}_g$, the subspace of all the surfaces with filling systole is a deformation retract of $\mathcal{M}_g$. In other words, in the Teichm\"uller space $\mathcal{T}_g$, the subspace of all the surfaces with filling systole is an equivariant deformation retract of $\mathcal{T}_g$ with respect to the action of the mapping class group. This subspace was later called \emph{Thurston spine}. One may denote the Thurston spine in $\mathcal{T}_g$ as $\mathcal{P}_g$. 
Today, Thurston's conjecture remains open. Over the years, a few results on the topology and geometry of Thurston spine were obtained, for example, a codimension-$2$ deformation retract containing $\mathcal{P}_g$ \cite{ji2014well}; local properties of $\mathcal{P}_g$ \cite{irmer2025morse}; cells in $\mathcal{P}_g$ with high dimensions \cite{bourque2020hyperbolic, fortier2024dimension, irmer2023small, mathieu2023estimating}. For the geometric aspect, see \emph{e.g.} \cite{anderson2016relative, gao2024shape}.

We consider Thurston's conjecture in a specific Teichm\"uller curve. A Teichm\"uller curve is a totally-geodesic embedding $i:\mathbb{H}^2\to \mathcal{T}_g$ such that $i(\mathbb{H}^2)$ covers an algebraic curve in $\mathcal{M}_g$. 
The Teichm\"uller curve we consider consists of the surfaces admitting an order-$(g+1)$ rotation. 
Induced by the rotation, it is treated as the embedding of the Teichm\"uller space of $4$-coned spheres $\mathcal{T}_{0,4} \cong \mathbb{H}^2$ into $\mathcal{T}_g$. One may denote this embedding as $\pi^{*}:\mathcal{T}_{0,4} \to \mathcal{T}_g$. 
For $\mathcal{T}_{0,4}$, the moduli space $\mathcal{M}_{0,4}$ is a $3$-punctured sphere (see Figure \ref{fig:spine}).
The image of $\pi^{*}(\mathcal{T}_{0,4})$ in $\mathcal{M}_g$ is double-branched covered by $\mathcal{M}_{0,4}$, by an order-$2$ rotation exchanging two of the three punctures and fixing the last one (see Lemma \ref{lem:modo}). 
Our main theorem (Theorem \ref{thm:main}) describes $\pi^{*}(\mathcal{T}_{0,4}) \cap \mathcal{P}_g$, the intersection between the Teichm\"uller curve and the Thurston spine. Denote by $q$ the covering $\mathcal{T}_g \to \mathcal{M}_g$. We have that
\begin{theorem}[= Theorem \ref{thm:main}]
In $\mathcal{M}_g$, $q(\pi^{*}(\mathcal{T}_{0,4}) \cap \mathcal{P}_g)$ consists of a circle around a puncture and a segment joining the circle and the singular point of $q(\pi^{*}(\mathcal{T}_{0,4}))$. 
The preimage of $q(\pi^{*}(\mathcal{T}_{0,4}) \cap \mathcal{P}_g)$ in $\mathcal{M}_{0,4}$ consists of two circles and a segment joining them, as is illustrated in Figure \ref{fig:spine}. 

 \label{thm:main_intro}
\end{theorem}

A direct corollary to our main theorem is 
\begin{corollary}[= Corollary \ref{cor:conj}]
	For the Teichm\"uller curve $\pi^{*}(\mathcal{T}_{0,4})$ and the Thurston spine $\mathcal{P}_g$, the intersection  $\pi^{*}(\mathcal{T}_{0,4}) \cap \mathcal{P}_g$ is an equivariant deformation retract of $\pi^{*}(\mathcal{T}_{0,4})$. 
 \label{cor:conj_intro}
\end{corollary}

\begin{figure}[htbp]
 \centering
 \includegraphics{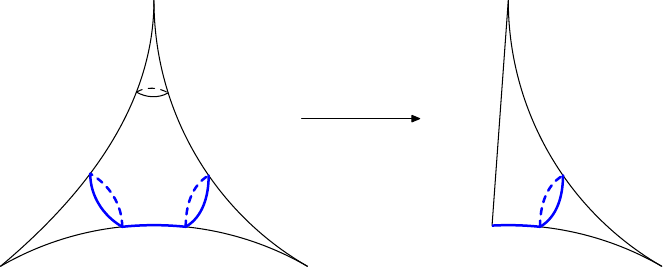}
 \caption{$q(\pi^{*}(\mathcal{T}_{0,4}) \cap \mathcal{P}_g)$ (right) and its preimage in $\mathcal{M}_{0,4}$ (left)}
 \label{fig:spine}
\end{figure}

By our construction, any element in a subgroup of the mapping class group $\modo_g$ is realized by an essential loop in $q(\mathcal{P}_g)$. This subgroup is a rank-$2$ free group and contains both reducible elements and pseudo-Anosov elements, see Corollary \ref{cor:loop}. To our knowledge, this is the first construction of an essential loop in $q(\mathcal{P}_g) \subseteq \mathcal{M}_g$.

The proof of Theorem \ref{thm:main} consists of two main new ingredients, describing a fundamental domain of the Teichm\"uller curve with respect to the Fenchel-Nielsen coordinate (Proposition \ref{prop:fundamental_curve}), classifying the systoles of the surfaces in this fundamental domain (Proposition \ref{prop:sys_cand}).

The aim of Section \ref{sec:fund_dom} is to describe a fundamental domain of the covering from $\pi^{*}(\mathcal{T}_{0,4}) \subseteq \mathcal{T}_g$ to its image in $\mathcal{T}_g/ \modo_g^{\pm}$ in terms of the Fenchel-Nielsen coordinates of $\mathcal{T}_{0,4}$, where $\modo^{\pm}_g$ is the extended mapping class group. 
This is divided into $2$ steps. First, we describe a fundamental domain of the action of the pure mapping class group $\pmodo_{0,4}$ on $\mathcal{T}_{0,4}$ (Proposition \ref{prop:fundamental}). Then we describe the covering from $\mathcal{M}_{0,4}\cong\mathcal{T}_{0,4}/\pmodo_{0,4}$ to the image of the concerned Teichm\"uller curve in $\mathcal{T}_g/\modo_g^{\pm}$ (Lemma \ref{lem:modo}). 
By Royden's theorem \cite{royden1971automorphisms}, the Teichm\"uller metric on $\mathcal{T}_{0,4}$ is the hyperbolic metric on $\mathbb{H}^2$, and the concerned groups act isometrically on $\mathcal{T}_{0,4}$. 
The main difficulty is that on $\mathcal{T}_{0,4}\cong \mathbb{H}^2$, we do not know the correspondence between the Fenchel-Nielsen coordinates and the coordinate of the hyperbolic plane. 

\begin{figure}[htbp]
 \centering
 \includegraphics{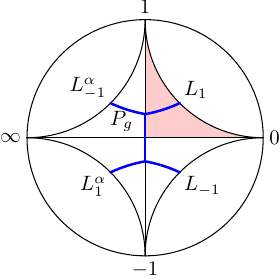}
 \caption{$\mathcal{P}_g$ in $\pi^{*}(F)$, the shaded part is $\pi^{*}(F_0)$}
 \label{fig:fund_dom_2}
\end{figure}

To obtain Proposition \ref{prop:fundamental}, one may take two pairs of curves in $\mathcal{T}_{0,4}$ described by the Fenchel-Nielsen coordinates, paired by the two generators of $\pmodo_{0,4}$ respectively. 
These two pairs are the $(L_{-1}, L_1)$ and $(L_{-1}^{\alpha}, L_1^{\alpha})$ illustrated in Figure \ref{fig:fund_dom_2}. 
One may show they are disjoint Teichm\"uller geodesics (Lemma \ref{lem:geod} and \ref{lem:disjoint}), bounding an ideal quadrilateral, which is the desired fundamental domain, denoted as $F$ and illustrated in Figure \ref{fig:fund_dom_2}. The most crucial ingredient is to show that these geodesics bound a quadrilateral. It depends on Lemma \ref{lem:limit}, implying that there is a geodesic $L_0$, whose one end is shared with $L_{-1}$ and $L_1$, and whose the other end is shared with $L_{-1}^{\alpha}$ and $L_1^{\alpha}$, and moreover, $L_0$ separates $L_1$, $L_{-1}$ and separates $L_1^{\alpha}$, $L_{-1}^{\alpha}$ respectively. The key ingredient to prove this key lemma is Minsky's product region theorem \cite{minsky1996extremal}.

In the second part of Section \ref{sec:fund_dom}, we prove Lemma \ref{lem:modo} by picking out the automorphisms of $\mathcal{M}_{0,4}$ that induce automorphisms of the concerned Teichm\"uller curve and thus prove Proposition \ref{prop:fundamental_curve}. The fundamental domain in Proposition \ref{prop:fundamental_curve} is the shaded triangle in Figure \ref{fig:fund_dom_2} and is denoted as $\pi^{*}(F_0)$. 

In Proposition \ref{prop:sys_cand}, the systoles on surfaces in $\pi^{*}(F_0)\cap \mathcal{P}_g$ are classified into $3$ multi-geodesics $\alpha$, $\beta$, and $\gamma$.
The proof is in $3$ steps. First, we show 
for any surface 
in the Teichm\"uller curve $\pi^{*}(\mathcal{T}_{0,4})$, its systole is among $4$ families of geodesics, and each family contains at most one systole up to the surface's symmetry (Lemma \ref{lem:family_sys_cand}) by taking advantage of the surface's symmetry. 
Then we restrict our discussion to the fundamental domain $\pi^{*}(F_0)$. We show that for any $X\in \pi^{*}(F_0)\cap \mathcal{P}_g$, its systole are among $4$ multi-geodesics $\alpha$, $\beta$, $\gamma$ and $\delta$ (Proposition \ref{prop:four_sys_cand}). 
At last, by a length comparison, we show $\delta$ is not a systole for surfaces in $\pi^{*}(F_0)\cap \mathcal{P}_g$ (Lemma \ref{lem:delta_not}), hence prove Proposition \ref{prop:sys_cand}. 

With the above preparations, we are ready to prove the main theorem. In $\pi^{*}(F_0)$, two curves are constructed consisting of the points with $\ell_{\beta} = \ell_{\gamma}$ and $\ell_{\alpha} = \ell_{\gamma}$ respectively in Lemma \ref{lem:u_implicit} and Lemma \ref{lem:u0_implicit}. These curves are treated as graphs of functions in terms of the Fenchel-Nielsen coordinate $\left( c, u\overset{\mathrm{def}}{=} \frac{t}{c} \right) $. For any point in subarcs $\ell_{\beta} = \ell_{\gamma} \le \ell_{\alpha}$ and $\ell_{\alpha} = \ell_{\gamma} \le \ell_{\beta}$, the surface has filling systoles hence is contained in $\mathcal{P}_g$. Using the implicit function theorem and the uniqueness of points in $\pi^{*}(F_0)$ with $\ell_{\alpha} = \ell_{\beta} = \ell_{\gamma} $ (Proposition \ref{prop:unique}), one may show these two arcs are connected and $\mathcal{P}_g\cap \pi^{*}(F_0)$ consists of these two arcs, which is the main theorem. 

In Section \ref{sec:pre}, the concerned surfaces are constructed and the Teichm\"uller curve is described in terms of its Fenchel-Nielsen coordinates. In Section \ref{sec:fund_dom}, Proposition \ref{prop:fundamental_curve} is proved and in Section \ref{sec:sys_cand}, Proposition \ref{prop:sys_cand} is proved. In Section \ref{sec:main}, the main theorem is proved. 

\subsection*{Acknowledgement} We would thank Prof. Kasra Rafi for helpful communications. 
Gao is supported by NSFC grant 12301082.

\addtocontents{toc}{\protect\setcounter{tocdepth}{2}}

\section{Basic construction} \label{sec:pre}
In this section, we construct the concerned surfaces admitting the order-$(g+1)$ rotation, then describe the Teichm\"uller curve consisting of all these surfaces in two Fenchel-Nielsen coordinates and give the coordinate change formulae between them. In the last part of this section, we introduce the Minsky product region theorem and the Gromov boundary in the special cases that are needed in the proof of Section \ref{sec:fund_dom}.
\subsection{Surface construction and symmetry}
In this subsection, we recall the construction of the concerned hyperbolic surface family, which was  originally constructed in \cite{gao2023maximal}, and was influenced by the construction of the symmetric topological surfaces in \cite{wang2015embedding}. For convenience, some figures in Section \ref{sub:family_cand}, \ref{sub:cand} are modified from figures in \cite{gao2023maximal}.

Let $n \ge 3$, 
take two isometric right-angled $2n$-gon admitting an order- $n$ rotation. These two $2n$-gons can be glued into an $n$-holed sphere with geodesic boundary by gluing their edges alternatively. One may call a boundary geodesic a \emph{cuff} and a glued edge a \emph{seam}.
This $n$-holed sphere still admits the order-$n$ rotation. 
By the rotation symmetry, every cuff has the same length and every seam has the same length. The geometry of the $n$-holed sphere is determined by its cuff length (or equivalently, its seam length).

Pick two copies of the $n$-holed spheres and glue them by pairing their cuffs. The requirements for the gluing are $(1)$ a cuff from one $n$-holed sphere is paired with a cuff from the other $n$-holed sphere; $(2)$ the order-$n$ rotation symmetry of an $n$-holed sphere can be extended to an order-$n$ rotation symmetry on the closed surface. Thus we obtain a genus $g=n-1$ hyperbolic surface $X$ admitting an order-$(g+1)$ rotation $\rho$. An example of $(X, \rho)$ is illustrated in Figure \ref{fig:exp_1_2}, where the endpoints of seams from one $n$-holed sphere are identified with the endpoints of seams from the other $n$-holed sphere. 

\begin{figure}[htbp]
 \centering
 \includegraphics{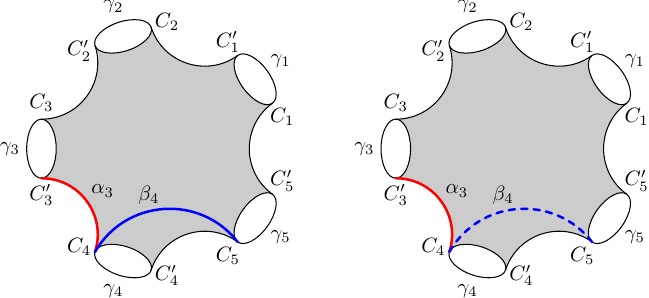}
 \caption{The surface $(X,\rho)$}
 \label{fig:exp_1_2}
\end{figure}

\subsection{Fenchel-Nielsen coordinate and the $\alpha$, $\beta$, $\gamma$ curves} \label{sub:fn}

The subspace in the Teichm\"uller space $\mathcal{T}_g$ of all the $(X, \rho)$ surfaces can be parametrized by a Fenchel-Nielsen coordinate $(c, t)$. 

Let $\gamma = \{ \gamma_1, \gamma_2, ..., \gamma_{g+1}\}$ be the cuffs of $X$. By the symmetry $\rho$, all cuffs have the same length. Let 
\[
 \ell_{\gamma_i} = 2c
\]
be the length parameter, where $i = 1,2,..., g+1$. The length is taken as $2c$ rather than $c$ for the convenience of following calculations. 

For the surface illustrated in Figure \ref{fig:exp_1_2} (denoted as $X_0$), the two seams joining $C'_iC_{i+1}$ in both $n$-holed spheres form a simple closed geodesic $\alpha_i$, for $i=1,2,...,g+1$ in $\mathbb{Z}/(g+1)\mathbb{Z}$. Let the twist parameter of $X_0$ be $0$, and the twist parameter of $D^t_{\gamma}(X_0)$ be $t$, where $D^t_{\gamma}$ is the Fenchel-Nielsen deformation along the multi-curve $\gamma$, with time $t$. 

Let $D_{\gamma_i}(\alpha_i)$ be the Dehn twist on $\alpha_i$ along $\gamma_i$. Define
\[
 \beta_i \overset{\mathrm{def}}{=} D_{\gamma_i}^{-1}(\alpha_i)
\]
for $i=1,2,..., g+1$. 

\subsection{Teichm\"uller space embedding}
For $(X,\rho)$ defined above, the quotient $O \overset{\mathrm{def}}{=}X/\langle \rho\rangle$ is an orbifold with spherical underlying space and four singular points of index $g+1$. 
Denote its Teichm\"uller space as $\mathcal{T}_{0,4}^g$. The orbifold covering
\[
 \pi: X \to O
\]
induces an embedding
\[
 \pi^{*}: \mathcal{T}_{0,4}^g \to \mathcal{T}_g. 
\]
The subspace of $\mathcal{T}_g$, consisting of surfaces admitting the  $\rho$-action is identified with $\pi^{*}(\mathcal{T}_{0,4}^g)$. The Fenchel-Nielsen coordinate defined in the above subsection is actually a Fenchel-Nielsen coordinate of  $\mathcal{T}_{0,4}^g$. One may omit the  '$g$' in $\mathcal{T}_{0,4}^g$ if not causing ambiguity. 

As illustrated in Figure \ref{fig:rotation} and \ref{fig:rotation_local}, for the orbifold covering $\pi:X \to O$, $\gamma_i$ is mapped to a simple closed geodesic $\gamma \subseteq O$, and the restriction $\pi|_{\gamma_i}$ is an isometry for $i =1, 2, ..., g+1$. Hence 
\[
 \ell_{\gamma_i}(X) = \ell_{\gamma}(O)
\]
for $i =1, 2, ..., g+1$. 
Similarly, $\alpha_X \overset{\mathrm{def}}{=} \{\alpha_1, \alpha_2,..., \alpha_{g+1} \}\subseteq X$ covers a simple closed geodesic $\alpha_O \subseteq O$, and 
\[
 \ell_{\alpha_i}(X) = \ell_{\alpha_O}(O). 
\]
for $i =1, 2, ..., g+1$. 
$\beta_X \overset{\mathrm{def}}{=} \{\beta_1, \beta_2,..., \beta_{g+1} \}\subseteq X$ covers a figure-eight closed geodesic $\beta_O \subseteq O$, and 
\[
 \ell_{\beta_i}(X) = \ell_{\beta_O}(O). 
\]
for $i =1, 2, ..., g+1$. 
One may omit the subscripts in $\alpha_X$, $\alpha_O$, $\beta_X$, $\beta_O$, if not causing ambiguity. 
    \begin{figure}[htbp]
 \centering
 \begin{subfigure}[htbp]{.45\textwidth}
 \begin{center}
  \includegraphics{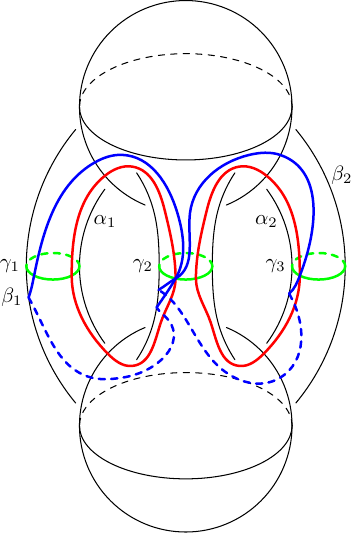}
 \end{center}
 \caption{The surface $X$}
 \label{fig:rotation_1}
 \end{subfigure}
 \begin{subfigure}[htbp]{.45\textwidth}
 \begin{center}
  \includegraphics{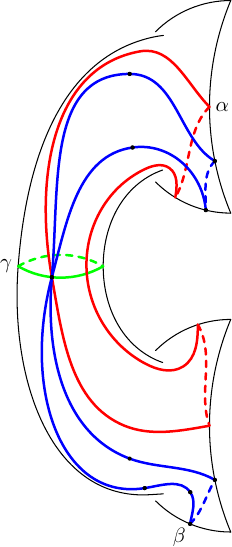}
 \end{center}
 \caption{The orbifold $O$}
 \label{fig:rotation_2}
 \end{subfigure}
 
 \caption{Curves $\alpha$, $\beta$, $\gamma$ in the surface}
 \label{fig:rotation}
\end{figure}

    \begin{figure}[htbp]
 \centering
 \begin{subfigure}[htbp]{.45\textwidth}
 \begin{center}
  \includegraphics{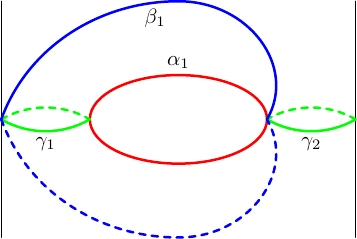}
 \end{center}
 \caption{The surface $X$}
 \label{fig:rotation_local_1}
 \end{subfigure}
 \begin{subfigure}[htbp]{.45\textwidth}
 \begin{center}
  \includegraphics{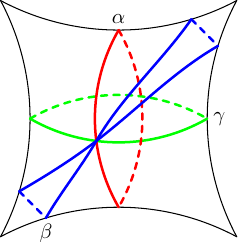}
 \end{center}
 \caption{The orbifold $O$}
 \label{fig:rotation_local_2}
 \end{subfigure}
 
 \caption{Curves $\alpha$, $\beta$, $\gamma$ in the surface (local picture)}
 \label{fig:rotation_local}
\end{figure}

On the orbifold $O = X / \langle \rho\rangle$, there is a simple closed geodesic (denoted as $\delta$), disjoint with $\beta$, see Figure \ref{fig:rotation_delta}. If $g$ is even, $\delta$ lifts to one simple closed geodesic $\delta_1$ in $X$ (see Figure \ref{fig:rotation_local_4} for the case $g=2$); if $g$ is odd, $\delta$ lifts to two simple closed geodesics $\delta_1$ and $\delta_2$ in $X$. 
With a little abuse of notation, on the surface $X$, one may let $\delta \overset{\mathrm{def}}{=} \{\delta_1\}$, if $g$ is even;  
$\delta \overset{\mathrm{def}}{=} \{\delta_1, \delta_2\}$, if $g$ is odd.  
    \begin{figure}[htbp]
 \centering
 \begin{subfigure}[htbp]{.45\textwidth}
 \begin{center}
  \includegraphics{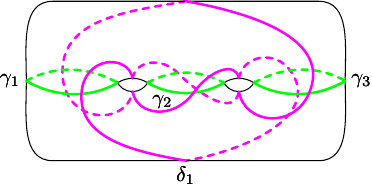}
 \end{center}
 \caption{The surface $X$}
 \label{fig:rotation_local_4}
 \end{subfigure}
 \begin{subfigure}[htbp]{.45\textwidth}
 \begin{center}
  \includegraphics{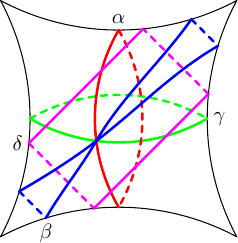}
 \end{center}
 \caption{The orbifold $O$}
 \label{fig:rotation_local_3}
 \end{subfigure}
 
 \caption{The curve $\delta$ in the surface}
 \label{fig:rotation_delta}
\end{figure}

\subsection{Mapping class group action}
The \emph{pure mapping class group} of a punctured surface consists of the mapping classes fixing each of its punctures.
For the $4$-punctured sphere, its pure mapping class group (denoted as $\pmodo_{0,4}$) is a rank-$2$ free group generated by the Dehn twists $D_{\alpha}$ and $D_{\gamma}$ (The curves $\alpha$ and $\gamma$ are illustrated in Figure \ref{fig:rotation_local_2}) \cite[Section 4.2.4]{farb2011primer}.

The moduli space $$\mathcal{M}_{0,4} \cong \mathcal{T}_{0,4} / \pmodo_{0,4}$$ is holomorphic to the $3$-punctured sphere $\hat{\mathbb{C}} \backslash \{0, 1, \infty\}$. By the uniformization theorem, there is a unique complete hyperbolic metric on the $3$-punctured sphere. By Royden's theorem, on the Teichm\"uller space, the Teichm\"uller metric coincides with the Kobayashi metric \cite[Theorem 3]{royden1971automorphisms}. On $\mathbb{H}^2$, the Kobayashi metric is the hyperbolic metric. Therefore
\begin{claim}
 The hyperbolic metric on the $3$-punctured sphere $\mathcal{M}_{0,4}$ is the Teichm\"uller metric on the moduli space $\mathcal{M}_{0,4}$. 
 \label{claim:dt_dk_dh}
\end{claim}

\begin{remark}
 Warning: This remark explains why $\pi^{*}(\mathcal{T}_{0,4})$ is a Teichm\"uller curve based on the language and some knowledge on Teichm\"uller curves may be known to experts in this area. But we do not provide preliminaries on Teichm\"uller curves because we do not need them in any part of this paper except this remark. Readers interested in Teichm\"uller curves may refer to \emph{e.g.} \cite{mcmullen2023billiards, moller2011teichmuller}. Readers only interested in hyperbolic geometry may skip this remark. 

 The embedding $\pi^{*}:\mathcal{T}_{0,4}\to \mathcal{T}_g$ is a Teichm\"uller curve that is induced by an $X\in \pi^{*}(\mathcal{T}_{0,4})$ and a holomorphic $1$-form $\omega$ on $X$ induced by the pair of multi-geodesics $\alpha$ and $\gamma$ that fills $X$. This is from two facts. First, $(X, \omega)$ is $\rho$-invariant and $\rho$ commutes with the $\mathrm{SL}_2(\mathbb{R})$ action on $(X,\omega)$, hence $\pi^{*}(\mathcal{T}_{0,4})$ is the $\mathrm{SL}_2(\mathbb{R})$ orbit of $X$ and is totally geodesic in $\mathcal{T}_g$. Second, $\mathrm{SL}(X, \omega)$ contains $\pmodo_{0,4}$ and is contained in $\modo_{0,4}$, and $\mathcal{T}_{0,4}/\pmodo_{0,4}$ and $\mathcal{T}_{0,4}/\modo_{0,4}$ are algebraic curves, hence $\pi^{*}(\mathcal{T}_{0,4})$ covers an algebraic curve in $\mathcal{M}_g$. 
The group $\pmodo_{0,4}\subseteq  \mathrm{SL}(X, \omega)$ because $D_{\alpha}$ and $D_{\gamma}$ acts on both $\mathcal{T}_{0,4}$ and $\pi^{*}(\mathcal{T}_{0,4})$. 
On the other hand, $  \mathrm{SL}(X, \omega)\subseteq\modo_{0,4}$ because 
any element of  $\mathrm{SL}(X, \omega)$ commutes with the $\rho$ action, hence is identified with an element in $\modo_{0,4}$. 
  The image of $\pi^{*}(\mathcal{T}_{0,4})$ in $\mathcal{M}_g$ is concretely descirbed in Lemma \ref{lem:modo}. 
\end{remark}

\subsection{Lengths of $\alpha$, $\beta$, $\gamma$ and $\delta$}
This subsection recalls length formulae of the geodesics $\alpha$, $\beta$, $\gamma$ and $\delta$ in the surfaces. All formulae are obtained in \cite{gao2023maximal}. 

For $X \in \pi^{*}(\mathcal{T}_{0,4})$ with the Fenchel-Nielsen coordinate $(c,t)$ defined in Section \ref{sub:fn}, let $s$ be the seam length of two isometric $n$-holed spheres forming $X$. When $c>0$ and $0 \le t \le c$, let $i = 1,2,...,g+1$ in (\ref{for:gamma})-(\ref{for:t=c}), then 
\begin{eqnarray}
 \cos \frac{\pi}{g+1} &=& \sinh \frac{c}{2} \sinh \frac{s}{2}  \text{\,\,\,\,\cite[(5-1)]{gao2023maximal}}; \label{for:seam_cuff} \\
\frac{\ell_{\gamma_i}(X)}{2} &=& c \text{\,\,\,\,\cite[(5-2), $\ell_{\gamma_i}(X) = 2|l_{CE}|$]{gao2023maximal}}; \label{for:gamma}\\
 \cosh \frac{\ell_{\alpha_i}(X)}{4} &=& \cosh \frac{s}{2} \cosh \frac{t}{2} \text{\,\,\,\,\cite[(5-3), $\ell_{\alpha_i}(X) = 4|l_{CD}|$]{gao2023maximal}}; \label{for:alpha}\\
 \cosh \frac{\ell_{\beta_i}(X)}{2} &=& \cosh s \cosh \frac{t}{2} \cosh \left( c- \frac{t}{2} \right) - \sinh \frac{t}{2} \sinh \left( c- \frac{t}{2} \right) \label{for:beta} \\
     & &\text{\,\,\,\,\cite[(5-5), $\ell_{\beta_i}(X) = 2|l_{C}|$]{gao2023maximal}}; \nonumber \\
 \cosh \frac{\ell_{\delta_1}(X)}{4g+4} &=& \cosh \frac{s}{2}\cosh \left(  \frac{c - t}{2} \right) \,\,\,\,\text{when $g$ is even }\label{for:delta_even}\\
          && \text{\cite[(5-4), $\ell_{\delta_1}(X) = 4(g+1) |l_{DE}|$]{gao2023maximal}}; \nonumber \\
 \cosh \frac{\ell_{\delta_j}(X)}{2g+2} &=& \cosh \frac{s}{2}\cosh \left( \frac{c - t}{2} \right) \,\,\,\,\text{when $g$ is odd}\label{for:delta_odd}\\
          &&\text{\cite[(5-4), $\ell_{\delta_j}(X) = 2(g+1) |l_{DE}|$]{gao2023maximal}, }\nonumber
\end{eqnarray}
for $j=1,2$. 

One may check by direct calculation, 
\begin{equation}
 \frac{\partial \ell_{\alpha_i}(X)}{\partial t} >0;\, \frac{\partial \ell_{\alpha_i}(X)}{\partial c} <0 \label{for:pd_alpha}
\end{equation}
and
\begin{equation}
 \frac{\partial \ell_{\beta_i}(X)}{\partial t} <0. \label{for:pd_beta}
\end{equation}
When $t=0$, 
\begin{equation}
 \ell_{\alpha_i}(X) < \ell_{\beta_i}(X);\, \ell_{\gamma_i}(X) < \ell_{\beta_i}(X). \label{for:t=0}
\end{equation}
When $t=c$,
\begin{equation}
 \ell_{\alpha_i}(X) > \ell_{\beta_i}(X). \label{for:t=c}
\end{equation}

\subsection{Fenchel-Nielsen coordinate change} \label{sub:coord_change}
We defined the Fenchel-Nielsen coordinate $(c, t)$ of $\mathcal{T}_{0,4}$ (\emph{resp. } $\pi^{*}(\mathcal{T}_{0,4})$) by letting $\gamma$ be the cuff and letting $\alpha$ be the simple closed (multi-)geodesic consisting of seams when $t = 0$. Conversely, if one let $\alpha$ be the cuff and let $\gamma$ be the simple closed (multi-)geodesic consisting of seams when the new twist parameter is $0$, then one may get a new Fenchel-Nielsen coordinate $(c_{\alpha}, t_{\alpha})$ of $\mathcal{T}_{0,4}$ (\emph{resp. } $\pi^{*}(\mathcal{T}_{0,4})$). 
One may give the coordinate change formulae between $(c, t)$ and $(c_{\alpha}, t_{\alpha})$. As the definition of $c$, let $c_{\alpha} \overset{\mathrm{def}}{=} \frac{1}{2}\ell_{\alpha_i}(X)$ for $i= 1,2, ..., g+1$. By (\ref{for:alpha}), we have that
\begin{equation}
 \cosh  \frac{c_{\alpha}}{2} = \cosh \frac{s}{2} \cosh \frac{t}{2}. \label{for:c_alpha}
\end{equation}
Let $s_{\alpha}$ and $t_{\alpha}$ be the seam length and twist parameter with respect to $\alpha$ respectively, then by the symmetry between $\alpha$ and $\gamma$, (\ref{for:seam_cuff}) and (\ref{for:c_alpha}), we have that
\begin{eqnarray}
 \cos \frac{\pi}{g+1} &=& \sinh \frac{c_{\alpha}}{2} \sinh \frac{s_{\alpha}}{2}; \label{for:seam_cuff_alpha} \\
 \cosh  \frac{c}{2} &=& \cosh \frac{s_{\alpha}}{2} \cosh \frac{t_{\alpha}}{2}. \label{for:t_alpha}
\end{eqnarray}

The formulae (\ref{for:seam_cuff}), (\ref{for:c_alpha}), (\ref{for:seam_cuff_alpha}), (\ref{for:t_alpha}) imply that
\begin{equation}
 t=0 \text{ iff } t_{\alpha}=0. 
 \label{for:tt_alpha_0}
\end{equation}
\begin{figure}[htbp]
 \centering
 \includegraphics{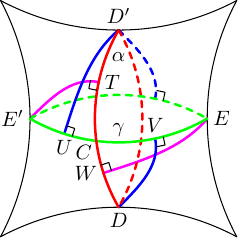}
 \caption{The seams}
 \label{fig:rotation_local_5}
\end{figure}

Observe that when $t\ne 0$, $\alpha$ and $\gamma$  are not perpenticular. Then the seams bound right-angled triangles with $\alpha$ and $\gamma$, where the acute angles between $\alpha$ and $\gamma$ are contained in these triangles. As is illustrated in Figure \ref{fig:rotation_local_5}, $C$ is an intersection point of $\alpha$ and $\gamma$. The angles $\angle D'CE'$ and $\angle DCE$ are acute. The segments $D'U$ and $DV$ are half seams with respect to the $(c,t)$ coordinate; they cobound the triangles $\triangle D'UC$ and $\triangle DVC$ with $\alpha$ and $\gamma$ respectively. The twist parameter $t$ is realized by $VU$.  On the other hand, the segments $E'T$ and $EW$ are half seams with respect to the $(c_{\alpha},t_{\alpha})$ coordinate, they cobound the triangles $\triangle E'TC$ and $\triangle EWC$ with $\alpha$ and $\gamma$ respectively. The twist parameter $t_{\alpha}$ is realized by $WT$. As the triangles $\triangle WCE$ and $\triangle VCD$ share an angle $\angle DCE$, when traveling along the piecewise geodesic $EWTE'$, one may turn right at $W$, while when traveling along the piecewise geodesic $DVUD'$, one may turn left at $V$. Hence 
if $t \ne 0$, then
\begin{equation}
 t\cdot t_{\alpha} <0. 
 \label{for:tt_alpha_neg}
\end{equation}

This coordinate change induces an order-$2$ orientation-preserving isometry on $\mathcal{T}_{0,4}$: 
\begin{eqnarray*}
 f: \mathcal{T}_{0,4} &\to & \mathcal{T}_{0,4} \\
 X &\mapsto& Y
\end{eqnarray*}
such that in the Fenchel-Nielsen coordinates $(c,t)$ and $(c_{\alpha}, t_{\alpha})$, 
\begin{equation}
 c(Y) = c_{\alpha}(X); \, t(Y) = t_{\alpha}(X). 
 \label{for:f_coord}
\end{equation}
The map $f$ is orientation-preserving and isometric because one can construct an orientation-preserving isometry from the surface $X$ to $f(X)$ exchanging the geodesics $\alpha$ and $\gamma$, by extending the isometry between $X \backslash \alpha$ and $f(X) \backslash\gamma$ to the whole surface. 
The order of $f$ is $2$ because by (\ref{for:seam_cuff}) (\ref{for:c_alpha}) (\ref{for:seam_cuff_alpha}) and (\ref{for:t_alpha}), the formula (\ref{for:f_coord}) is equivalent to 
\[
 c_{\alpha}(Y) = c(X); \, t_{\alpha}(Y) = t(X). 
\]

\subsection{Minsky's product regions theorem}

Minsky's product regions theorem \cite[Theorem 6.1]{minsky1996extremal} implies the following theorem for $\mathcal{T}_{0,4}$. 

\begin{theorem}
 For the Fenchel-Nielsen coordinate $(c,t)$ of $\mathcal{T}_{0,4}$, identify $\mathcal{T}_{0,4}$ and $\mathbb{H}^2$ by 
 \begin{eqnarray*}
  \phi:\mathcal{T}_{0,4} & \to & \mathbb{H}^2 \\
  (c,t) &\mapsto&   \frac{t}{c} + \frac{i}{c}. 
 \end{eqnarray*}
 Let $d_{\mathcal{T}}$ be the Teichm\"uller distance on $\mathcal{T}_{0,4}$ and $d_{\mathbb{H}}$ be the hyperbolic distance on $\mathbb{H}^2$. For a sufficiently small $\epsilon >0$, for any $X, Y \in \mathcal{T}_{0,4}$ satisfying $\ell_{\gamma}(X), \ell_{\gamma}(Y) < \epsilon$, there is a constant $C >0$ such that
 \[
 |d_{\mathcal{T}}(X,Y) - d_{\mathbb{H}}(\phi(X), \phi(Y))| < C. 
 \]
 \label{thm:minsky}
\end{theorem}

\subsection{Gromov boundary of $\mathbb{H}^2$}
For the hyperbolic plane $\mathbb{H}^2$, its \emph{Gromov boundary} consists of equivalent classes of geodesic rays. For two geodesic rays parametrized by their arc lengths $r_1, r_2: [0,+\infty)\to \mathbb{H}^2$, $r_1\sim r_2$ iff there is a $C>0$ such that 
\[
 \limsup\limits_{t\to \infty} d_{\mathbb{H}}(r_1(t), r_2(t)) <C. 
\]
For the upper half plane model, the Gromov boundary of $\mathbb{H}^2$ is identified with $\mathbb{R}\cup \{\infty\}$. One may denote it as $\partial_{\infty} \mathbb{H}^2$. 

\section{Fundamental domain}
\label{sec:fund_dom}
In this section, we describe a fundamental domain of the covering from $\pi^{*}(\mathcal{T}_{0,4})$ to its image in $\mathcal{T}_g/\modo_g^{\pm}$, which is denoted as $\pi^{*}(F_0)$. In the first subsection, a fundamental domain of $\mathcal{T}_{0,4}$ under the the action of $\pmodo_{0,4}=\langle D_{\alpha}, D_{\gamma} \rangle$ (denoted as $F$) is described in Proposition \ref{prop:fundamental}. In the second subsection, the fundamental domain $\pi^{*}(F_0)$ is described in Proposition \ref{prop:fundamental_curve}. 
\subsection{The fundamental domain $F$}
In this subsection, we describe a fundamental domain of $\mathcal{T}_{0,4}$ under the action of $\pmodo_{0,4}=\langle D_{\alpha}, D_{\gamma} \rangle$. 
\begin{proposition}
 For $\mathcal{T}_{0,4}$ with the Fenchel-Nielsen coordinates $(c,t)$ and $(c_{\alpha}, t_{\alpha})$ described in Section \ref{sec:pre}, we have that
 \[
  F \overset{\mathrm{def}}{=}\{ O \in \mathcal{T}_{0,4} | |t(O)| \le c(O), |t_{\alpha}(O)| \le c_{\alpha}(O) \}
 \]
 is a fundamental domain of $\mathcal{T}_{0,4}$ under the action of $\pmodo_{0,4} = \langle D_{\alpha}, D_{\gamma} \rangle$. 

 \label{prop:fundamental}
\end{proposition}
We prove Proposition \ref{prop:fundamental} in steps. 
First, we define six curves in $\mathcal{T}_{0,4}$. These curves are described by the Fenchel-Nielsen coordinates $(c,t)$ and $(c_{\alpha}, t_{\alpha})$. 
\begin{eqnarray*}
 L_{-1} & \overset{\mathrm{def}}{=} & \left\{ O \in \mathcal{T}_{0,4} \left| \frac{t}{c} = -1 \right. \right\}; \\
 L_{0} & \overset{\mathrm{def}}{=} & \left\{ O \in \mathcal{T}_{0,4} | t = 0 \right\}; \\
 L_{1} & \overset{\mathrm{def}}{=} & \left\{ O \in \mathcal{T}_{0,4} \left| \frac{t}{c} = 1 \right. \right\}; \\
 L_{-1}^{\alpha} & \overset{\mathrm{def}}{=} & \left\{ O \in \mathcal{T}_{0,4} \left| \frac{t_{\alpha}}{c_{\alpha}} = -1 \right. \right\}; \\
 L_{0}^{\alpha} & \overset{\mathrm{def}}{=} & \left\{ O \in \mathcal{T}_{0,4} | t_{\alpha} = 0 \right\}; \\
 L_{1}^{\alpha} & \overset{\mathrm{def}}{=} & \left\{ O \in \mathcal{T}_{0,4} \left| \frac{t_{\alpha}}{c_{\alpha}} = 1 \right. \right\}. 
\end{eqnarray*}

\begin{lemma}
 In $\mathcal{T}_{0,4}$, the six curves
 $L_{-1}$, $L_0$, $L_1$, $L_{-1}^{\alpha}$, $L_0^{\alpha}$, $L_1^{\alpha}$
 are geodesics with respect to Teichm\"uller metric. 
 \label{lem:geod}
\end{lemma}
\begin{proof}
 For $O \in \mathcal{T}_{0,4}$, one may consider the reflection (denoted as $r_0$), mapping $O$ to its mirror image. On the Teichm\"uller space $\mathcal{T}_{0,4}$, $r_0$ maps $(c,t)$ to $(c,-t)$ and is an isometry with respect to Teichm\"uller metric. The fixed-point set of $r_0$ in $\mathcal{T}_{0,4}$ is the curve $L_0$. Hence $L_0$ is a Teichm\"uller geodesic. 

The Dehn twist $D_{\gamma}$ maps $(c,t)$ to $(c, t+2c)$ (Recall $\ell_{\gamma}(O) =2c$. ). As $D_{\gamma} \in \pmodo_{0,4}$, $D_{\gamma}$ is an isometry on $\mathcal{T}_{0,4}$ with respect to Teichm\"uller metric by Royden's theorem (see \emph{e.g.}\cite{royden1971automorphisms}). Then $L_1$ is the fixed-point set  of the isometry
\begin{eqnarray*}
 D_{\gamma} \circ r_0: \mathcal{T}_{0,4} &\to& \mathcal{T}_{0,4} \\
 (c,t) & \mapsto & \left( {c}, -t+2c \right), 
\end{eqnarray*}
hence is a Teichm\"uller geodesic. 

Similarly, $L_{-1}$ is the fixed-point set  of the isometry $D_{\gamma}^{-1} \circ r_0$, hence is a Teichm\"uller geodesic. 

The curves  $L_{-1}^{\alpha}$, $L_{0}^{\alpha}$, $L_{1}^{\alpha}$ are proved to be Teichm\"uller geodesics by showing that they are the fixed-point sets of the isometries $D_{\alpha}^{-1} \circ r_0$, $r_0$ and $D_{\alpha} \circ r_0$ respectively. 
The proof is complete. 
\end{proof}

The geodesics $L_0$ and $L_0^{\alpha}$ are the fixed-point set of the same isometry $r_0$, which implies the following lemma. 

\begin{lemma}
 The geodesic $L_0$ coincides with the geodesic $L_0^{\alpha}$. On this geodesic, $c\to 0$ iff $c_{\alpha} \to \infty$; $c\to \infty$ iff $c_{\alpha} \to 0$. 
 \label{lem:0_coicide}
\end{lemma}
\begin{proof}
 The geodesics $L_0 = L^{\alpha}_0$ follows from (\ref{for:tt_alpha_0}). Combining (\ref{for:seam_cuff}), (\ref{for:c_alpha}) and (\ref{for:seam_cuff_alpha}), one may have, when $t=t_{\alpha}=0$,
 \begin{equation}
 \sinh \frac{c}{2}\sinh \frac{c_{\alpha}}{2} = \cos \frac{\pi}{g+1}. 
 \label{for:cc_alpha_t0} 
 \end{equation}
Thus $c\to 0$ iff $c_{\alpha} \to \infty$; $c\to \infty$ iff $c_{\alpha} \to 0$. The proof is complete. 

\end{proof}

The geodesics $L_{-1}$, $L_0$, $L_1$ are parametrized by $c$, where $c\in (0,+\infty)$. Consider the geodesic rays of $L_{-1}$, $L_0$, $L_1$ restricting $c\in (0,1]$. 
\begin{lemma}
 The geodesic rays $L_{-1}|_{c \le 1}$, $L_0|_{c \le 1}$, $L_1|_{c \le 1}$ are the same point of the Gromov boundary of $\mathcal{T}_{0,4}$. 
 \label{lem:limit}
\end{lemma}
\begin{proof}
 The main ingredient of the proof is Theorem \ref{thm:minsky}. Some symbols for example $\phi$, $C$, $d_{\mathcal{T}}$, $d_{\mathbb{H}}$ are from the statement of Theorem \ref{thm:minsky}. 
Take the Fenchel-Nielsen coordinate $(c,t)$ of $\mathcal{T}_{0,4}$.
 For a sufficiently small $c_0 \in (0,1)$,  let
 \begin{eqnarray*}
  l_j: (0,c_0] &\to& \mathcal{T}_{0,4} \\
  c &\mapsto& (c,j\cdot c)
 \end{eqnarray*}
 be the map corresponding to the concerned geodesic ray $L_j|_{c \le c_0}$, where $j= -1, 0, 1$. 
 Let 
 \begin{eqnarray*}
  f_j: (0,c_0] & \to & [0,+\infty) \\
  c& \mapsto & x = f_j(c) 
 \end{eqnarray*}
 be the arc length reparametrization of $l_j$. Namely
 \begin{eqnarray*}
  l_j\circ f_j^{-1}: [0,+\infty) & \to & \mathcal{T}_{0,4} \\
         x & \mapsto &  l_j\circ f_j^{-1}(x)= l_j(c) = (c, j\cdot c)
 \end{eqnarray*}
 is the geodesic ray parametrized by arc length with respect to the Teichm\"uller metric $d_{\mathcal{T}}$. 

 WLOG, one may consider the rays $L_0|_{c \le c_0}$ and $L_1|_{c \le c_0}$. 
 The endpoints of these rays in $\mathcal{T}_{0,4}$ are $$l_0\circ f^{-1}_0(0) = (c_0, 0)\text{ and }l_1\circ f^{-1}_1(0) = (c_0, 1\cdot c_0)$$ respectively. For these points, one may have $$\phi\circ l_0\circ f^{-1}_0(0) = \frac{i}{c_0}\text{ and }\phi\circ l_1\circ f^{-1}_1(0) = 1+\frac{i}{c_0}.$$ 
 For any $x>0$, consider $$l_0\circ f^{-1}_0(x) = (f_0^{-1}(x),0) \text{ and }l_1\circ f^{-1}_1(x) = (f_1^{-1}(x),1\cdot f_1^{-1}(x)) .$$
 For these points, one may have that
\begin{equation}
\phi\circ l_0\circ f^{-1}_0(x) =  \frac{i}{f_0^{-1}(x)} \text{ and }\phi\circ l_1\circ f^{-1}_1(x) = 1+ \frac{i}{f_1^{-1}(x)}. 
\label{for:hcoord}
\end{equation}
 Since $c_0$ is sufficiently small, by Theorem \ref{thm:minsky}, there is a $C>0$ such that, on the geodesic ray $L_0|_{c \le c_0}$, one may have that
 \[
  |d_{\mathcal{T}}(l_0\circ f^{-1}_0(x), l_0\circ f^{-1}_0(0)) - d_{\mathbb{H}} (\phi\circ(l_0\circ f^{-1}_0)(x), \phi\circ(l_0\circ f^{-1}_0)(0)) |  < C 
 \]
namely,
 \[
 \left|x - \log \frac{c_0}{f_0^{-1}(x)} \right| <C. 
 \]
On the geodesic ray $L_1|_{c \le c_0}$, one may have that
 \[
 |d_{\mathcal{T}}(l_1\circ f^{-1}_1(x), l_1\circ f^{-1}_1(0)) - d_{\mathbb{H}} (\phi\circ(l_1\circ f^{-1}_1)(x), \phi\circ(l_1\circ f^{-1}_1)(0)) |  < C 
 \]
 namely,
 \[
 \left|x - \log \frac{c_0}{f_1^{-1}(x)} \right| <C.   
 \]
 Eliminating $c_0$ and $x$, one may get that
 \begin{equation}
  \left| \log \frac{f^{-1}_1(x)}{f^{-1}_0(x)} \right| < 2C . 
  \label{for:lipshitz}
 \end{equation}
 Again by Theorem \ref{thm:minsky}, one may have that
 \begin{equation}
  |d_{\mathcal{T}}(l_0\circ f_0^{-1}(x), l_1\circ f_1^{-1}(x)) - d_{\mathbb{H}}(\phi\circ l_0\circ f_0^{-1}(x), \phi\circ l_1\circ f_1^{-1}(x))| < C. 
  \label{for:dtdh}
 \end{equation}
 Then, one may get that
 \begin{eqnarray}
  & & d_{\mathbb{H}}(\phi\circ l_0 \circ f_0^{-1}(x), \phi\circ l_1 \circ f_1^{-1}(x))\nonumber \\
  &=& d_{\mathbb{H}}\left( \frac{i}{f_0^{-1}(x)}, 1+\frac{i}{f_1^{-1}(x)} \right) \nonumber\,\,\,\,\text{(by (\ref{for:hcoord}))}\\
  & \le & d_{\mathbb{H}}\left( \frac{i}{f_0^{-1}(x)}, 1+\frac{i}{f_0^{-1}(x)} \right) + d_{\mathbb{H}}\left( 1+\frac{i}{f_0^{-1}(x)}, 1+\frac{i}{f_1^{-1}(x)} \right) \nonumber\\
  & \le & 1+2C. \,\,\,\,\text{(by }f_0^{-1}(x) < c_0 <1 \text{ and (\ref{for:lipshitz}) )} \label{for:dh}
 \end{eqnarray}
 Combining (\ref{for:dtdh}) and (\ref{for:dh}), one may get that, for any $x>0$, 
 \[
 d_{\mathcal{T}}(l_0\circ f_0^{-1}(x), l_1\circ f_1^{-1}(x)) \le 1+3C.  
 \]

 Therefore the geodesic rays $L_{0}|_{c \le 1}$ and $L_{1}|_{c \le 1}$ represent the same point in the Gromov boundary. 
The same proof holds for $L_{-1}|_{c \le 1}$ and  $L_0|_{c \le 1}$. The proof is complete.

\end{proof}

One may denote the point on the Gromov boundary represented by $L_{-1}|_{c \le 1}$, $L_0|_{c \le 1}$ and $L_1|_{c \le 1}$ as $\infty$. For the other end of $L_0$, namely the geodesic ray $L_0|_{c \ge 1}$, one may denote the point on the Gromov boundary represented by this ray as $0$. By Lemma \ref{lem:0_coicide}, the end of $L_0|_{c \ge 1}$ is the end of $L^{\alpha}_0|_{c_{\alpha} \le 1}$. By Lemma \ref{lem:limit}, $0$ is also represented by $L^{\alpha}_1|_{c_{\alpha} \le 1}$ and $L^{\alpha}_{-1}|_{c_{\alpha} \le 1}$. 

Recall that Teichm\"uller metric on $\mathcal{T}_{0,4}$ is the hyperbolic metric and $\pmodo_{0,4}$ isometrically acts on $\mathcal{T}_{0,4}$. By the Nielsen-Thurston classification (see \emph{e.g.} \cite{farb2011primer, fathi2021thurston}), $D_{\gamma}$, $D_{\alpha}$ are reducible elements of $\pmodo_{0,4}$. Also $D_{\gamma}(L_{-1}|_{ c\le 1}) = L_1|_{ c\le 1}$ hence $D_{\gamma}(\infty) = \infty$, and $D_{\alpha}(L_{-1}^{\alpha}|_{c_{\alpha} \le 1}) = L_1^{\alpha}|_{c_{\alpha} \le 1}$ hence $D_{\alpha}(0) = 0$. Hence
\begin{claim}
 The Dehn twists $D_{\gamma}$, $D_{\alpha}$ are parabolic elements in $$\pmodo_{0,4}\subseteq \mathrm{Iso}(\mathcal{T}_{0,4})=\mathrm{Iso}(\mathbb{H}^2).$$ In particular, $0$ and $\infty$ are the fixed points of $D_{\alpha}$ and $D_{\gamma}$ respectively. 
 \label{claim:parabolic}
\end{claim}

\begin{lemma}
 The four geodesics $L_{1}$, $L_{-1}$, $L_{1}^{\alpha}$, $L_{-1}^{\alpha}$ are pairwise disjoint. 
 \label{lem:disjoint}
\end{lemma}
\begin{proof}
 By Lemma \ref{lem:limit}, $L_{-1}$ is disjoint with $L_1$ and $L_{-1}^{\alpha}$ is disjoint with $L_1^{\alpha}$ inside $\mathcal{T}_{0,4}$. Now we show $L_1$ and $L_1^{\alpha}$ are disjoint. 

 Suppose for contradiction that $p \in L_1\cap L_1^{\alpha}$. Then at $p$ the Fenchel-Nielsen coordinates $(c, t)$ and $(c_{\alpha}, t_{\alpha})$ satisfy that
 \begin{equation}
  \frac{t}{c} = 1;\text{\,\,} \frac{t_{\alpha}}{c_{\alpha}} = 1.  
  \label{for:tc=1}
 \end{equation}
 Then inserting (\ref{for:tc=1}) into (\ref{for:c_alpha}) and (\ref{for:t_alpha}) and multiplying the two formulae, one may get that
 \[
  \cosh \frac{s}{2} \cosh \frac{s_{\alpha}}{2} =1.
 \]
 This implies $s= s_{\alpha} = 0$, which is impossible. 
 The same proof holds for the pairs $(L_1, L_{-1}^{\alpha})$, $(L_{-1}, L_{1}^{\alpha})$ and $(L_{-1}, L_{-1}^{\alpha})$ because $\cosh(t) = \cosh(-t)$. 
 The proof is complete. 
\end{proof}

Now we are ready to prove Proposition \ref{prop:fundamental}. 
\begin{proof}[Proof of Proposition \ref{prop:fundamental}]
 Let 
\begin{eqnarray*}
F_1 &\overset{\mathrm{def}}{=}& \left\{O \in \mathcal{T}_{0,4} \left| \left|\frac{t}{c}\right| \le 1 \right.\right\};\\
F_2 &\overset{\mathrm{def}}{=}& \left\{O \in \mathcal{T}_{0,4} \left| \left|\frac{t_{\alpha}}{c_{\alpha}}\right| \le 1 \right.\right\}. 
\end{eqnarray*}
Then $F_1$ is a fundamental domain of $D_{\gamma}$ and $F_2$ is a fundamental domain of $D_{\alpha}$. 

Recall $L_0$ is the geodesic joining the fixed point of $D_{\gamma}$ (\emph{i.e.} $\infty$) and the fixed point of $D_{\alpha}$ (\emph{i.e.} $0$) in the Gromov boundary. 
By Lemma \ref{lem:limit} and Claim \ref{claim:parabolic}, we have that
\begin{equation}
L_0 \subseteq F_1 \cap F_2. 
\label{for:contain}
\end{equation}
Thus $F_1 \cap F_2 \ne \emptyset$. Moreover, $F_1 \nsubseteq F_2$ and $F_2 \nsubseteq F_1$ (see Figure \ref{fig:fund_dom}). 
This is because $L_0$ separates $\mathbb{H}^2$. The geodesics $L_1$ and $L_{-1}$ lie on different sides of $L_0$ ; $L_1^{\alpha}$ and $L_{-1}^{\alpha}$ also lie on different sides of $L_0$. Let $\partial F_j \subseteq \partial_{\infty}\mathbb{H}^2$ consist of geodesic rays contained in $F_j$ for $j= 1,2$. Then $\infty$ is in the interior of $\partial F_2$, but $\infty$ is an isolated point of $\partial F_1 $. 
Therefore $F_2 \nsubseteq F_1$. 
Conversely, $0$ is in the interior of $\partial F_1$, but $0$ is an isolated point of $\partial F_2$. Therefore $F_1 \nsubseteq F_2$.
\begin{figure}[htbp]
 \centering
 \includegraphics{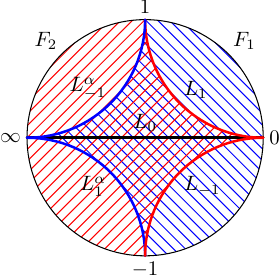}
 \caption{The fundamental domain}
 \label{fig:fund_dom}
\end{figure}

Then $F_1 \cap F_2$ is a convex region in $\mathcal{T}_{0,4}$ bounded by four geodesics, as illustrated in Figure \ref{fig:fund_dom}. By Lemma \ref{lem:disjoint}, these geodesics are pairwise disjoint. For the generators $D_{\alpha}$ and $D_{\gamma}$ of $\pmodo_{0,4}$, $D_{\alpha}$ pairs two edges of $F_1 \cap F_2$, namely the edges $L_{-1}^{\alpha}$ and $L_{1}^{\alpha}$; $D_{\gamma}$ pairs the other two edges of $F_1 \cap F_2$, namely the edges $L_{-1}$ and $L_{1}$. Thus $F=F_1 \cap F_2$ is a fundamental domain of $\pmodo_{0,4} = \langle D_{\alpha}, D_{\gamma} \rangle$. The proof is complete.
\end{proof}
Since $\mathcal{M}_{0,4}$ is a sphere with three punctures, one may get
\begin{claim}
  The fundamental domain $F$ is an ideal quadrilateral. 
 \label{claim:cusp}
\end{claim}
By Claim \ref{claim:cusp} and (\ref{for:tt_alpha_neg}), one may find that $L_1$ and $L_{-1}^{\alpha}$ share an end on $\partial_{\infty}\mathbb{H}^2$ and denote it as $1$. Similarly, $L_{-1}$ and $L_{1}^{\alpha}$ share an end on $\partial_{\infty}\mathbb{H}^2$ and one may denote it as $-1$. See Figure \ref{fig:fund_dom} for an illustration. 

\subsection{The fundamental domain $F_0$}
Consider the Teichm\"uller curve $\pi^{*}(\mathcal{T}_{0,4}) \subseteq \mathcal{T}_g$. Let $\Pi$ be the covering from $\pi^{*}(\mathcal{T}_{0,4})$ to its image in $\mathcal{T}_g/ \modo_g^{\pm}$. 
Based on the fundamental domain $F$ of $\pmodo_{0,4}$, one may show that 
\begin{proposition}
 A fundamental domain of the covering $\Pi$ is $\pi^{*}(F_0) \subseteq \pi^{*}(\mathcal{T}_{0,4})$, where 
\begin{equation}
 F_0  \overset{\mathrm{def}}{=} \{ O\in F | t(O) \ge 0, c(O) \le c_{\alpha}(O)\}. 
\label{for:quad_fund} 
\end{equation}
The domain $\pi^{*}(F_0)$ is the shaded triangle in Figure \ref{fig:fund_dom_2}.
 \label{prop:fundamental_curve}
\end{proposition}

To prove this proposition, we first introduce a reflection $r_1$ on $\mathcal{T}_{0,4}$ and the Teichm\"uller geodesic fixed by it (denoted as $L_{-1,1}$). 
Recall the order-$2$ rotation $f$ on $\mathcal{T}_{0,4}$ defined in Section \ref{sub:coord_change}. 
Define $r_1$ as $r_0\circ f$. Then with respect to the Fenchel-Nielsen coordinates $(c,t)$ and $(c_{\alpha}, t_{\alpha})$, $r_1$ maps $X\in\mathcal{T}_{0,4}$ to $Y\in\mathcal{T}_{0,4}$ such that
\[
 c(Y) = c_{\alpha}(X); \, t(Y) = -t_{\alpha}(X). 
\]
The Teichm\"uller geodesic fixed by $r_1$ is 
\[
 L_{-1,1} \overset{\mathrm{def}}{=} \left\{ X\in \mathcal{T}_{0,4} | c(X) = c_{\alpha}(X)\right\}. 
\]
On $L_{-1,1}$, $t(X) = -t_{\alpha}(X)$ follows from $c(X) = c_{\alpha}(X)$ by (\ref{for:c_alpha}) and (\ref{for:t_alpha}).
One may describe the ends of $L_{-1,1}$ on $\partial_{\infty}\mathbb{H}^2$. 
\begin{lemma}
The ends of $L_{-1,1}$ on $\partial_{\infty}\mathbb{H}^2$ are $-1$ and $1$. Hence $L_{-1,1} \subseteq F$.  
 \label{lem:l_pm1}
\end{lemma}
\begin{proof}
 To prove this lemma, it is sufficient to verify the following two statements. 
 \begin{enumerate}
  \item The geodesics $L_{-1,1}$ and $L_0$ intersect once. 
  \item The geodesics $L_{-1,1}$ is disjoint with $L_{1}$, $L_{-1}$, $L^{\alpha}_{1}$ and $L^{\alpha}_{-1}$. 
 \end{enumerate}
 The statement $(1)$ follows immediately from (\ref{for:cc_alpha_t0}). 
 To show $(2)$, by the symmetry of $r_0$ and $r_1$, it is sufficient to prove that
 \[
  L_{-1,1} \cap L_1 = \emptyset. 
 \]
 Suppose for contradiction that there is a point $X \in L_{-1,1} \cap L_1$. Then consider the Fenchel-Nielsen coordinates $(c,t)$ and $(c_{\alpha}, t_{\alpha})$ at $X$, one may have
 \[
  c = t = c_{\alpha}. 
 \]
 Then by (\ref{for:c_alpha}), one may have
 \[
  s = 0, 
 \]
 which is impossible. The proof is complete. 
\end{proof}
One may check that the rotation $f$ is an order-$2$ rotation of the ideal quadrilateral $F$. 
Thus $f$ induces an involution $\bar{f}$ of $\mathcal{M}_{0,4}$.
Similarly, denote the reflections in $\mathcal{M}_{0,4}$ induced by $r_0$ and $r_1$ as $\bar{r_0}$ and $\bar{r_1}$ respectively. 
One may show that 
\begin{lemma}
 \begin{enumerate}
  \item The image of $\pi^{*}(\mathcal{T}_{0,4}) $ in $\mathcal{M}_g$ is isomorphic to $\mathcal{M}_{0,4}/ \left<\bar{f} \right>$.
  \item The image of $\pi^{*}(\mathcal{T}_{0,4}) $ in $\mathcal{T}_g/ \modo^{\pm}_g$ is isomorphic to $\mathcal{M}_{0,4}/\left< \bar{r_0}, \bar{r_1} \right>$.
 \end{enumerate}
   \label{lem:modo}
\end{lemma}
\begin{proof}
 $(1)$ Consider the ideal quadrilateral $F$ in Figure \ref{fig:fund_dom}. 
 One may check that its 
  four ideal points $0$,  $\infty$, $1$, $-1$ correspond to the surfaces pinching $\alpha$, $\gamma$, $\delta$ and $D_{\gamma}^{-1}(\delta)$ respectively by (\ref{for:gamma}), (\ref{for:alpha}), (\ref{for:delta_odd}) and (\ref{for:delta_even}). 

 For the $3$-punctured sphere $\mathcal{M}_{0,4} \cong \mathcal{T}_{0,4}/ \left<D_{\alpha}, D_{\gamma} \right>$, two of its punctures correspond to $0$ and $\infty$ respectively. The last puncture is glued from $1$ and $-1$, and one may denote it as $\pm 1$. 
The automorphism group of $\mathcal{M}_{0,4}$ is isomorphic to the permutation group $S_3$ that permutes the three punctures.

If a permutation does not fix the puncture $\pm1$, 
then it does not induce an automorphism of $\pi^{*}(\mathcal{T}_{0,4}) \subseteq \mathcal{T}_g$, 
since in $\pi^{*}(\mathcal{T}_{0,4})$, $\delta$ consists of $1$ or $2$ geodesics; while both $\alpha$ and $\gamma$ consist of $g+1$ geodesics, where $g \ge 2$. 
Thus in the automorphism group of $\mathcal{M}_{0,4}$, the only non-trivial element that induces an automorphism of $\pi^{*}(\mathcal{T}_{0,4})$ is the automorphism $\bar{f}$, which permutes $0$ and $\infty$, fixes $\pm1$. Thus $(1)$ is proved. 

$(2)$ The statement $(2)$ follows directly from the facts $\left<\bar{r_0}, \bar{f} \right> = \left<\bar{r_0}, \bar{r_1} \right>$ and $\modo^{\pm}_g/\modo_g \cong \left<\bar{r_0} \right>$. The proof is complete. 

\end{proof}
\begin{proof}[Proof of Proposition \ref{prop:fundamental_curve}]
 It follows directly from Lemma \ref{lem:modo} $(2)$, because the domain $\left\{ O\in F | t \ge 0 \right\} $ is a fundamental domain of $r_0$ and the domain $\left\{ O\in F | c \le c_{\alpha} \right\} $ is a fundamental domain of $r_1$. 
\end{proof}

\section{Classification of systoles}
\label{sec:sys_cand}

For $X\in \mathcal{T}_g$, a \emph{systole} of $X$ is one of the shortest geodesics in $X$. One may denote by $\mathrm{S}(X)$ the set of systoles on  $X$. In $X$, a set of simple closed geodesics is \emph{filling} if the geodesics cut $X$ into polygonal disks. The \emph{Thurston spine} in $\mathcal{T}_g$ consists of the surfaces $X$, whose $\mathrm{S}(X)$ is filling, and one may denote it as $\mathcal{P}_g$. 
The main result of this section is 

\begin{proposition}
 For $X \in \mathcal{P}_g \cap \pi^{*}(F_{0})$, one may have
 \[
  \mathrm{S}(X) \subseteq \{ \alpha, \beta, \gamma\}. 
 \]
 \label{prop:sys_cand}
\end{proposition}

The proof of this proposition is divided into two parts. First we show $\mathrm{S}(X) \subseteq \{ \alpha, \beta, \gamma, \delta\}$ (Proposition \ref{prop:four_sys_cand}). Then we show $\delta \notin \mathrm{S}(X)$ (Lemma \ref{lem:delta_not}). 

\subsection{Symmetry of the concerned surface and quotient orbifolds}

For $X \in \pi^{*}(\mathcal{T}_{0,4}^{g})$, recall that $X$ is glued from two $(g+1)$-holed spheres admitting the order-$(g+1)$ rotation $\rho$. Besides the rotation $\rho$,  there are two order-$2$ rotations acting on $X$. One is the hyperelliptic involution, exchanging the two $(g+1)$-holed spheres, denoted as $\tau$. The other involution (denoted as $\sigma$) can be restricted to an involution on one of the $(g+1)$-holed spheres, exchanging the two $(2g+2)$-gons (see Figure \ref{fig:exp_1_4}). 

\begin{figure}[htbp]
 \centering
 \includegraphics{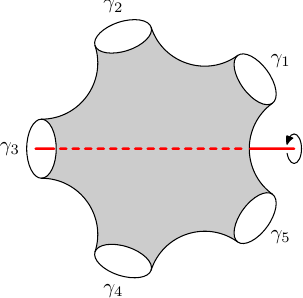}
 \caption{The rotation $\sigma$}
 \label{fig:exp_1_4}
\end{figure}

We describe three orbifolds induced by (subgroups of) $\langle \rho, \sigma, \tau \rangle$, which turns out to be useful in the following proofs. 

Let's denote by $O_1$ the orbifold $X/ \langle \tau \rangle$. As $\tau$ is a hyperelliptic involution, this orbifold has a spherical underlying space and $(2g+2)$ singular points of index $2$. One may denote the singular points as $C_1$, $C_2$, ..., $C_{2g+2}$ respectively, as illustrated in Figure \ref{fig:cyclic_cover_n_4}. 
\begin{figure}[htbp]
 \centering
 \includegraphics{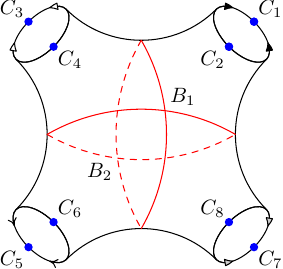}
 \caption{The orbifold $O_1$}
 \label{fig:cyclic_cover_n_4}
\end{figure}

The next orbifold is $O_1/ \langle \rho \rangle = X/ \left<\rho, \tau \right>$, denoted as $O_2$. This orbifold has a spherical underlying space and four singular points. Two of the points have indices $2$ (denoted as $C_1$ and $C_2$ by a little abuse of notation), while two of them have indices $g+1$ (denoted as $B_1$ and $B_2$), as illustrated in Figure \ref{fig:22nn_222n_1}. The preimages of $C_1$ and $C_2$ in $O_1$ consist of the $2g+2$ index-$2$ singular points of $O_1$, while the preimages of $B_1$ and $B_2$ are two regular points respectively. 
\begin{figure}[htbp]
 \centering
 \includegraphics{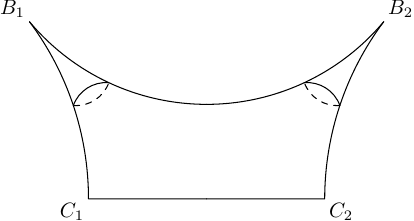}
 \caption{The orbifold $O_2$}
 \label{fig:22nn_222n_1}
\end{figure}

The third orbifold is $O_2/ \langle \sigma \rangle = X/ \left<\rho, \tau, \sigma \right>$, denoted as $O_3$. This orbifold has a spherical underlying space and four singular points. Three of the points have indices $2$ (denoted as $C$, $D$ and $E$), while one of them has index $g+1$ (denoted as $B$), as illustrated in Figure \ref{fig:22nn_222n_2}. In $O_2$, the preimage of $C$ consists of the singular points $C_1$ and $C_2$, and the preimage of $B$ consists of the singular points $B_1$ and $B_2$. On the other hand, in $O_2$ the preimages of $D$ and $E$ are two regular points respectively. 
\begin{figure}[htbp]
 \centering
 \includegraphics{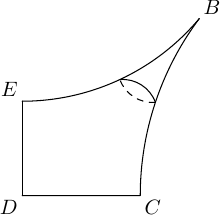}
 \caption{The orbifold $O_3$}
 \label{fig:22nn_222n_2}
\end{figure}

We list the orbifolds and some coverings among them. The coverings $p_1$ and $p_3$ will be repeatedly used in the following subsections. 
\begin{equation*}
\begin{tikzcd}
X \ar[r, "p_1"]  \ar[rrr, bend right=6mm, "p_3"]
& O_1=X/ \left<\tau \right> \ar[r ] 
& O_2=X/ \left<\rho, \tau \right> \ar[r] 
& O_3=X/ \left< \rho, \tau, \sigma \right>. 
\end{tikzcd}
\end{equation*}

\subsection{Families of possible systoles}\label{sub:family_cand}
In this subsection, 
we define four families of simple closed geodesics on $X$, according to their images in $O_3$. We show that $\mathrm{S}(X)$ is contained in these four families (Lemma \ref{lem:family_sys_cand}). 

A well-known fact about systoles is 
\begin{claim}
 Any systole is simple. Two systoles intersect at most once. 
 \label{claim:once_intersect}
\end{claim}

Consider the double-branched cover $p_1:X \to O_1$. 
For a simple closed geodesic $\alpha \subseteq X$, its image $p_1(\alpha) \subseteq O_1$ is either a closed geodesic (if not passing through the singular points of $O_1$), or an arc joining two singular points (if passing through the singular points of $O_1$). 
Given $\alpha, \beta \subseteq X$, if $p_1(\alpha)$ and $p_1(\beta)$ are both closed geodesics or both arcs, then one may say they are \emph{of the same type}. 
\begin{lemma}
 \begin{enumerate}
  \item For a simple closed geodesic $\alpha \subseteq X$, if $p_1(\alpha)$ intersects itself at some regular points of $O_1$, then $\alpha$ is not a systole. 
\item For the simple closed geodesics of the same type, $\alpha, \beta \subseteq X$ of the same type, if $p_1(\alpha)$ intersects $p_1(\beta)$ at some regular points of $O_1$, then either $\alpha$ or $\beta$ is not a systole. 
 \end{enumerate}
 \label{lem:intersect}
\end{lemma}
\begin{proof}

 As $p_1:X\to O_1$ is a double branched cover, 
 for a simple closed geodesic $\alpha \subseteq X$, if $p_1(\alpha)$ has a self-intersection at a regular point, 
then the preimage
  $p_1^{-1}(p_1(\alpha))$ has at least two self-intersection points. Moreover, $p_1^{-1}(p_1(\alpha))$ consists of either a closed geodesic (which is impossible as $\alpha$ is simple) or a pair of simple closed geodesics with equal length. By Claim \ref{claim:once_intersect}, $(1)$ is proved.

 For simple closed geodesics $\bar{\alpha}, \bar{\beta} \subseteq O_1$, if they intersect, then they intersect at least twice, because the underlying space of $O_1$ is a sphere and hence $\bar{\alpha}$, $\bar{\beta}$ are separating on $O_1$. The preimages $p_1^{-1}(\bar{\alpha})$, $p_1^{-1}(\bar{\beta})$ consist of one or two simple closed geodesic(s). No matter the preimages $p_1^{-1}(\bar{\alpha})$, $p_1^{-1}(\bar{\beta})$ has one or two component(s), any component of $p_1^{-1}(\bar{\alpha})$ intersects a component of $p_1^{-1}(\bar{\beta})$ at least twice, see Figure \ref{fig:double_cover}. By Claim \ref{claim:once_intersect}, $(2)$ holds when $p_1(\alpha)$ and $p_1(\beta)$ are both closed geodesics. 

    \begin{figure}[htbp]
 \centering
 \begin{subfigure}[htbp]{\textwidth}
 \begin{center}
  \includegraphics{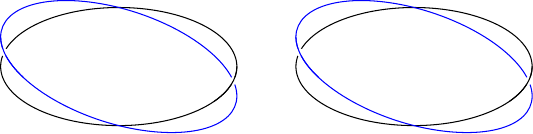}
 \end{center}
 \caption{Case 1}
 \label{fig:double_cover_2}
 \end{subfigure}

 \begin{subfigure}[htbp]{.45\textwidth}
 \begin{center}
  \includegraphics{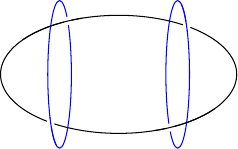}
 \end{center}
 \caption{Case 2}
 \label{fig:double_cover_3}
 \end{subfigure}
 \begin{subfigure}[htbp]{.45\textwidth}
 \begin{center}
  \includegraphics{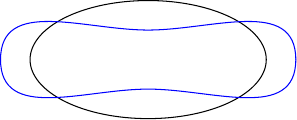}
 \end{center}
 \caption{Case 3}
 \label{fig:double_cover_4}
 \end{subfigure}
 
 \caption{Double covers of a pair of twice-intersecting simple curves}
 \label{fig:double_cover}
\end{figure}

For simple geodesic arcs $\bar{\alpha}, \bar{\beta} \subseteq O_1$ joining a pair of singular points, if they intersect at a regular point, then their preimages $p_1^{-1}(\bar{\alpha})$, $p_1^{-1}(\bar{\beta})$ are a pair of simple closed geodesics intersecting twice, see Figure \ref{fig:double_cover_arc}. By Claim \ref{claim:once_intersect}, $(2)$ holds when $p_1(\alpha)$ and $p_1(\beta)$ are both arcs. The proof is complete. 

\begin{figure}[htbp]
 \centering
 \begin{subfigure}[htbp]{.45\textwidth}
 \begin{center}
 \includegraphics{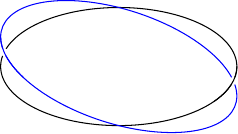}
 \end{center}
 \caption{The double cover}
 \label{fig:double_cover_1}
 \end{subfigure}
 \begin{subfigure}[htbp]{.45\textwidth}
 \begin{center}
  \includegraphics{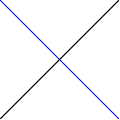}
 \end{center}
 \caption{The pair of arcs}
 \label{fig:double_cover_2s}
 \end{subfigure}
 \caption{Double branched cover of a pair of arcs}
 \label{fig:double_cover_arc}
\end{figure}

\end{proof}

Consider the branched cover $p_3:X\to O_3$. One may show that
\begin{lemma}
 For $\alpha\in \mathrm{S}(X)$, 
 \begin{enumerate}
  \item $p_3(\alpha)$ is simple. 
  \item $p_3(\alpha)$ is an arc joining two singular points. The two endpoints may be the same. 
  \item No endpoint of $p_3(\alpha)$ is $B$. 
  \item Endpoints of $p_3(\alpha)$ cannot be both $D$ or both $E$. 
 \end{enumerate}
 \label{lem:sys_cand_p3}
\end{lemma}

\begin{proof}
 This lemma relies on Lemma \ref{lem:intersect}. 

 $(1)$ Suppose for contradiction that $p_3(\alpha)$ is not simple. Then its preimage on $O_1$ is either a geodesic with self-intersections, or some geodesics with the same type and equal length intersecting at regular points. By Lemma \ref{lem:intersect}, $\alpha \notin \mathrm{S}(X)$. 

 $(2)$ To prove $(2)$ is to prove $p_3(\alpha)$ is not a simple closed geodesic. Recall that $O_3$ has a spherical underlying space and $4$ singular points. Then every simple closed geodesic on $O_3$ separates $O_3$, and on each side of this simple closed geodesic, there are two singular points. As there are three index-$2$ singular points on $O_3$, on one of the two sides, both singular points are index-$2$. But on any hyperbolic orbifold, no simple closed geodesic bounds a disk with two index-$2$ singular points by the Gauss-Bonnet theorem for orbifolds. Therefore $(2)$ holds. 

 $(3)$ Suppose for contradiction that an endpoint of $p_3(\alpha)$ is $B$. Notice that the preimage of $B$ on $O_1$ is a pair of regular points (illustrated as $B_1$, $B_2$ in Figure \ref{fig:cyclic_cover_n_4}). On the other hand the preimage of $p_3(\alpha)$ in $O_1$ is either a geodesic with self-intersections, or some geodesics with the same type and equal length intersecting at $B_1$ or $B_2$. By Lemma \ref{lem:intersect}, $\alpha \notin \mathrm{S}(X)$. 

 $(4)$ WLOG, suppose for contradiction that both endpoint of $p_3(\alpha)$ is $E$. Recall the preimage of $E$ in $O_2$ is a regular point. Then the preimage of $p_3(\alpha)$ is a non-simple geodesic, see Figure \ref{fig:self_int}, and its preimage in $O_1$ is either a geodesic with self-intersections or some geodesics with the same type and equal length intersecting at some regular points. By Lemma \ref{lem:intersect}, $\alpha \notin \mathrm{S}(X)$. 
 The proof is complete. 

 \begin{figure}[htbp]
  \centering
  \includegraphics{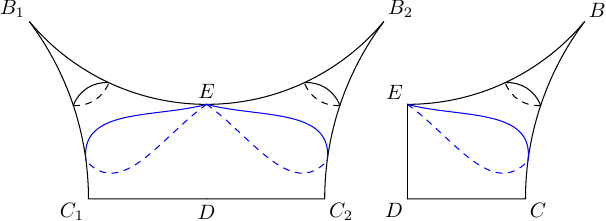}
  \caption{An arc with both endpoints at $E$ and its preimage in $O_2$}
  \label{fig:self_int}
 \end{figure}

\end{proof}

One may define some families of geodesics in $X$. Let $\mathcal{C}(X)$ be the set of simple closed geodesics on $X$. We define that
\begin{eqnarray*}
 L_{CD} &\overset{\mathrm{def}}{=} & \{ \alpha \subseteq \mathcal{C}(X) | \,p_3(\alpha) \text{ is a simple arc joining } CD \}; \\
 L_{CE} &\overset{\mathrm{def}}{=} & \{ \alpha \subseteq \mathcal{C}(X) | \,p_3(\alpha) \text{ is a simple arc joining } CE \}; \\
 L_{DE} &\overset{\mathrm{def}}{=} & \{ \alpha \subseteq \mathcal{C}(X) | \,p_3(\alpha) \text{ is a simple arc joining } DE \}; \\
 L_{C} &\overset{\mathrm{def}}{=} & \{ \alpha \subseteq \mathcal{C}(X) | \,p_3(\alpha) \text{ is a simple arc whose both endpoints are } C \}. 
\end{eqnarray*}

For the families, we show that
\begin{lemma}
 For $X\in \pi^{*} (\mathcal{T}_{0,4})$, $\mathrm{S}(X)\cap L_{CD}$, $\mathrm{S}(X)\cap L_{CE}$, $\mathrm{S}(X)\cap L_{C}$, $\mathrm{S}(X)\cap L_{DE}$ contains at most one geodesic up to the action of $\left<\rho, \sigma, \tau \right>$. 
 \label{lem:family_unique}
\end{lemma}
\begin{proof}
 For geodesics $\alpha$ and $\alpha'$, $p_3(\alpha)\ne p_3(\alpha')$ iff they have different  orbits  under the action of $\left<\rho, \sigma, \tau \right>$. 
 For $\alpha, \alpha'\in L_{CD}$ with $p_3(\alpha)\ne p_3(\alpha')$,  by definition, $p_3(\alpha)$ and $p_3(\alpha')$ pass through $D\in O_3$. The preimage of $D$ is a regular point in  $O_2$ and $O_1$. Hence the preimage of $p_3(\alpha)$ and $p_3(\alpha')$ in $O_1$ intersect at some regular points of $O_1$. By Lemma \ref{lem:intersect}, either $\alpha$ or  $\alpha'$ is not a systole of $X$. 

 The same proof holds for geodesics in $L_{CE}$. 

For $\alpha, \alpha'\in L_{DE}$ with $p_3(\alpha)\ne p_3(\alpha')$, by definition, $p_3(\alpha)$ and $p_3(\alpha')$ pass through $D, E\in O_3$. Their preimages on $O_2$ and $O_1$ are simple closed geodesics intersecting at regular points that are preimages of $D$ and $E$. By Lemma \ref{lem:intersect}, either $\alpha$ or  $\alpha'$ is not a systole of $X$. 

By Lemma \ref{lem:intersect}, if $\alpha, \alpha'\in L_{C} \cap \mathrm{S}(X)$, then $p_3(\alpha)$ and $p_3(\alpha')$ are disjoint except at $C$ by Lemma \ref{lem:intersect}. The geodesic $p_3(\alpha)$ separates $O_3$ into two coned disks (see Figure \ref{fig:222nssc_1}). In their interior, one contains the two singular points $D$, $E$ (denoted as $D_2$), the other contains one singular point $B$ (denoted as $D_1$). The other geodesic $p_3(\alpha')$ is contained in either disk. In $D_1$, any incontractible simple curve is parallel to $\partial D_1$, namely $p_3(\alpha)$. In $D_2$, an incontractible, non-$\partial $ parallel, simple closed curve based on $C$ separates $D$ and $E$. But a geodesic in this homotopy class is a geodesic in $L_{CD}$ or $L_{CE}$. Thus in both cases, $p_3(\alpha)=  p_3(\alpha')$. Therefore, the systole in $L_C$ is unique. The proof is complete. 

\begin{figure}[htbp]
 \centering
 \includegraphics{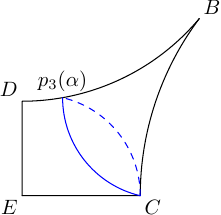}
 \caption{$p_3(\alpha)$ in $O_3$}
 \label{fig:222nssc_1}
\end{figure}

\end{proof}

To conclude, we show that
\begin{lemma}
 If $\alpha \in \mathrm{S}(X)$, then $\alpha \in L_{CD} \cup L_{CE} \cup L_{DE}\cup L_{C}$. Moreover, in each of the four families $L_{CD}$, $ L_{CE}$, $L_{DE}$ and $L_{C}$, $\mathrm{S}(X)$ contains at most one geodesic up to the action of $\left<\rho, \sigma, \tau \right>$. 
 \label{lem:family_sys_cand}
\end{lemma}
\begin{proof}
 This comes directly from Lemma \ref{lem:sys_cand_p3} and \ref{lem:family_unique}. 
\end{proof}

This gives a classification of systoles on $X$. 

\subsection{Restricting to the fundamental domain}
\label{sub:cand}

This subsection mainly shows that 
\begin{proposition}
 For $X\in \pi^{*}(F_{0})\cap \mathcal{P}_g$, $\mathrm{S}(X) \subseteq \{\alpha, \beta, \gamma, \delta\}$. 
 \label{prop:four_sys_cand}
\end{proposition}

To see this, first we describe the preimages of $C, D, E \in O_3$ on $X$. The preimage $p_3^{-1}(C)$ consists of Weierstrass points of $X$. As illustrated in Figure \ref{fig:exp_1_2_n}, $C_i$ and $C_i'$ are Weierstrass points of $X$ where $i = 1,2,...,g+1$, and these points are on the cuffs $\gamma = \{\gamma_1, \gamma_2, ..., \gamma_{g+1} \}$. The preimages $p_3^{-1}(D)$ or $p_3^{-1}(E)$ is either the mid-points of $C_iC'_i$ on the cuffs, or the mid-points of $C'_iC_{i+1}$ for $i = 1,2,...,g+1$. Also illustrated in Figure \ref{fig:exp_1_2_n}, the geodesic $\alpha_i$ consists of geodesics joining $C'_iC_{i+1}$ for $i = 1,2,...,g+1$. WLOG, assume the preimage of $D$ consists of the mid-points of $C_iC'_i$ and the preimage of $E$ consists of the mid-points of $C'_iC_{i+1}$. Hence $\gamma\in L_{CD}$ and $\alpha\in L_{CE}$. 

\begin{figure}[htbp]
 \centering
 \includegraphics{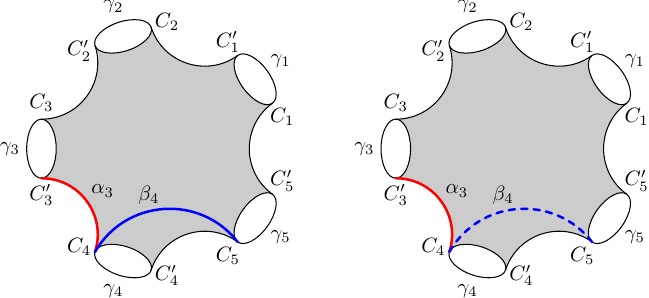}
 \caption{The surface $X$}
 \label{fig:exp_1_2_n}
\end{figure}

Recall that $\beta_i = D_{\gamma_i}^{-1}(\alpha_i)$, then as illustrated in Figure \ref{fig:exp_1_2_n}, $\beta_i$ consists of the geodesics joining $C_i, C_{i+1}$, intersects the $\alpha$-geodesics and $\gamma$-geodesics at $C_i$ and $C_{i+1}$. Hence $\beta\in L_C$. 
Recall $\delta \subseteq X$ is the preimage of the simple closed geodesic disjoint with $\pi(\beta)$, where $\pi: X \to O = X/\left<\rho \right>$. Since $\delta$ is disjoint with $\beta$, $\delta$ does not pass through $C$. Since $\pi(\delta)$ is simple, $\delta$ passes through $D$ and $E$. Hence $\delta\in L_{DE}$. To conclude
\begin{claim}
 The geodesics $\alpha\in L_{CE}$, $\beta \in L_{C}$, $\gamma\in L_{CD}$, $\delta \in L_{DE}$. 
 \label{claim:curve_in_family}
\end{claim}

The following lemma implies for $X \in \mathcal{P}_g \cap \pi^{*}(\mathcal{T}_{0,4})$, $\mathrm{S}(X)$ contains at least one geodesic in  $L_{CD}$ or $L_{CE}$, which is an essential observation to prove Proposition \ref{prop:four_sys_cand}. Moreover, it is the reason why we can exclude $\delta$ in  Lemma \ref{lem:delta_not}. 

\begin{lemma}
 If $\mathrm{S}(X) \subseteq L_C\cup L_{DE}$, then $X \notin \mathcal{P}_g$. 
 \label{lem:bd_not_spine}
\end{lemma}
\begin{proof}
 If $\mathrm{S}(X) \subseteq L_C$ or $\mathrm{S}(X) \subseteq L_{DE}$, then by Lemma \ref{lem:family_unique}, $\mathrm{S}(X)$ contains exactly one geodesic up to the $\left<\rho, \sigma, \tau \right>$ action. Thus $X \notin \mathcal{P}_g$. 

 Then we consider the case $\mathrm{S}(X) = \left\{ \eta, \eta' \right\} $, where $\eta\in L_C $ and $\eta'\in L_{DE} $. 
 First we assume $p_3(\eta)\cap p_3(\eta') = \emptyset$, as illustrated in Figure \ref{fig:222nssc_3}.
 Then  the preimage of $p_3(\eta')$ in $X$ consists of disjoint simple geodesics, and the preimage of $p_3(\eta')$ is disjoint with the preimage of $p_3(\eta)$. 
 Therefore $\left\{ \eta, \eta' \right\} $ does not fill $X$ and $X\notin \mathcal{P}_g$. 
 \begin{figure}[htbp]
  \centering
  \includegraphics{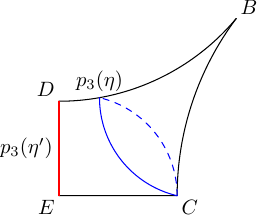}
  \caption{The disjoint $p_3(\eta)$ and $p_3(\eta')$}
  \label{fig:222nssc_3}
 \end{figure}

 For the case $p_3(\eta) \cap p_3(\eta')\ne \emptyset$, observe that both $p_3(\eta)$ and $p_3(\eta')$ are arcs joining singular points, while any component of the preimage of $p_3(\eta)$ or $p_3(\eta')$ in $X$ is a simple closed geodesic. Moreover, $p_3(\eta)$ intersects $p_3(\eta')$ at regular points of $O_3$. Therefore, for any component of the preimage of $p_3(\eta)$ in $X$, there is a component of the preimage of $p_3(\eta')$ in $X$ such that these two closed geodesics intersect at least twice, see Figure \ref{fig:double_cover_arc}. Thus, in this case, either $\eta$ or $\eta'$ is not a systole. The proof is complete. 
\end{proof}

Recall the covering $\Pi$ from the Teichm\"uller curve $\pi^{*}(\mathcal{T}_{0,4}) \subseteq \mathcal{T}_g$ to its image in $\mathcal{T}_{g}/\modo_{g}^{\pm}$. Recall by Lemma \ref{lem:modo}, $\Pi(\pi^{*}(\mathcal{T}_{0,4}))$ is isomorphic to $\mathcal{M}_{0,4}/ \left< \bar{r_0}, \bar{r_1} \right>$. 
Let $$\mathcal{P}_{0,4}^{\pm}\overset{\mathrm{def}}{=} \Pi(\mathcal{P}_g \cap \pi^{*}(\mathcal{T}_{0,4})).  $$

We are ready to show
\begin{lemma}
 There is a map $s: \mathcal{P}_{0,4}^{\pm}\to \mathcal{P}_g \cap \pi^{*}(F_0)$ such that $s \circ \Pi = \mathrm{id}_{\mathcal{P}_{0,4}^{\pm}}$. Moreover, for any $X\in \pi^{*}(F_0)$, $\gamma\in \mathrm{S}(X)$. 
 \label{lem:section}
\end{lemma}
By Proposition \ref{prop:fundamental_curve}, one can see the map $s$ is bijective. 

\begin{proof}
 First, we construct the map $s$. 
 For a surface $X \in \pi^{*}(\mathcal{T}_{0,4})\cap \mathcal{P}_g$, by Lemma \ref{lem:family_sys_cand} and \ref{lem:bd_not_spine}, there is a systole $\eta$, contained in $L_{CD}\cup L_{CE}$. 
 Therefore, for $Y \in \mathcal{P}^{\pm}_{0,4}$, there is a multi-geodesic consisting of $g+1$ systoles, cutting $Y$ into two $(g+1)$-holed spheres. Each sphere admits the $\rho$ action. We let this multi-geodesic be the cuff of $s(Y)$ and denote it as $\gamma = \left\{ \gamma_1, \gamma_2, ..., \gamma_{g+1} \right\} $. In particular, when considering the Fenchel-Nielsen coordinate $(c,t)$ of $s(Y)$, we have
 \[
  c \overset{\mathrm{def}}{=} \frac{1}{2} \ell_{\gamma_i}(s(Y)), 
 \]
 for any $i = 1,2,...,g+1$. 

 The next thing is to construct the multi-geodesic $\alpha$ and the twist parameter of $s(Y)$. 
 The geodesic $\gamma$ cuts $s(Y)$ into two $(g+1)$-holed spheres. Let $s_i$ and $s_i'$ be the shortest common perpenticular between $\gamma_i$ and $\gamma_{i+1}$ in the two $(g+1)$-holed spheres respectively, see Figure \ref{fig:tt_alpha_1}. The points $U_i, V_i$ are endpoints of $s_i$, while $U_i', V_i'$ are endpoints of $s_i'$. The points $U_i, U_i', V_{i-1}, V_{i-1}'$ lie on the cuff $\gamma_i$. 
 We first consider the case $U_i$ does not coincide with $V_{i-1}'$, and thus $U_i'$ does not coincide with $V_{i-1}$. 
 Let $U_iU_i'$ be the shortest arc contained in $\gamma_i$ joining $U_iU_i'$; let $V_{i-1}V_{i-1}'$ be the shortest arc contained in $\gamma_i$ joining $V_{i-1}V_{i-1}'$. Define $\alpha_i$ to be the geodesic isotopic to the piecewise geodesic $U_iU_i'\cup s_i'\cup V_i'V_i \cup s_i$. Define the twist parameter $t$ of $s(X)$ as the length of $U_iU_i'$ or $V_iV_i'$. One may choose a representative of $Y$ with a proper orientation to guarantee $t \ge 0$. 

 \begin{figure}[htbp]
  \centering
  \includegraphics{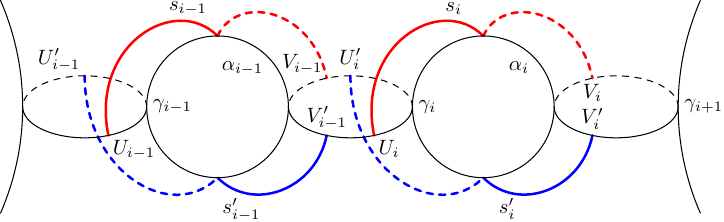}
  \caption{The twist parameter of $s(Y)$}
  \label{fig:tt_alpha_1}
 \end{figure}

 Therefore, we have constructed the map $s$ by constructing the Fenchel-Nielsen coordinate $(c,t)$ of the image $s(Y)$. 
 Since $\gamma$ is the systole of $X$, we have $\ell_{\gamma_i}(s(Y)) \le \ell_{\alpha_i}(s(Y))$ for $i=1,2,...,g+1$. Since $U_iU_i'$ and $V_{i-1}V_{i-1}'$ share the same length, and the assumption that $U_i$ does not coincide with $V_{i-1}'$, and $U_i'$ does not coincide with $V_{i-1}$, then $t < c$. 

For the case $U_i$ coincides with $V_{i-1}'$, and thus $U_i'$ coincides with $V_{i-1}$, one may assign $t=c$, and construct $\alpha$ as the former case. Notice the difference to the former case is that there are two choices of $U_iU_i'$ and $V_iV_i'$ such that $\alpha$ is $\rho$-invariant. But only one of them satisfies $t \ge 0$, and we take this choice. 

 The only thing left to check is $|t_{\alpha}| \le c_{\alpha}$. 

 If $t=0$, by (\ref{for:tt_alpha_0}), then $t_{\alpha} = 0$, so we only consider the case $t >0$. 
 Consider the seam $U'_{i-1}V'_{i-1}$ between $\gamma_{i-1}$ and $\gamma_i$ ; the seam $W_iT_i$ in the coordinate $(c_{\alpha}, t_{\alpha})$ between $\alpha_{i-1}$ and $\alpha_i$, see Figure \ref{fig:tt_alpha_2}. The mid-point $D_{i-1}$ of $U'_{i-1}V'_{i-1}$ is on $\alpha_{i-1}$, while the mid-point $E_{i}$ of $W_{i}T_{i}$ is on $\gamma_{i}$ by symmetry. The point $C_{i-1}$ is the intersection point of $\gamma_{i-1}$ and $\alpha_{i-1}$, $C'_{i-1}$ is the intersection point of $\alpha_{i-1}$ and $\gamma_i$, $C_i$ is the intersection point of $\gamma_i$ and $\alpha_i$. The point pairs $(C_{i-1}, C_{i-1}')$, $(C_{i-1}', C_i)$ separate each of $\alpha_{i-1}$, $\gamma_i$ into two arcs with equal length respectively. The point $E_i$ is the mid-point of one of the two arcs joining $C_{i-1}'C_i$; the point $D_{i-1}$ is the mid-point of one of the two arcs joining $C_{i-1}C_{i-1}'$. 

    \begin{figure}[htbp]
 \centering
 \begin{subfigure}[htbp]{\textwidth}
 \begin{center}
  \includegraphics{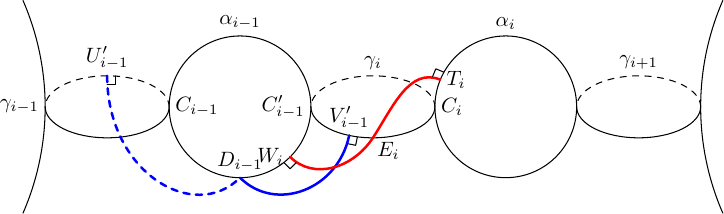}
 \end{center}
 \caption{The triangles in the surface}
 \label{fig:tt_alpha_2}
 \end{subfigure}

 \begin{subfigure}[htbp]{\textwidth}
 \begin{center}
  \includegraphics{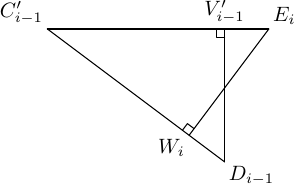}
 \end{center}
 \caption{The triangles}
 \label{fig:tt_alpha_3}
 \end{subfigure}

 \caption{Estimating $|t_{\alpha}|$}
 \label{fig:tt_alpha}
\end{figure}

In the right-angled triangles $\triangle C_{i-1}' D_{i-1} V'_{i-1}$ and $\triangle C'_{i-1}E_iW_i$ illustrated in Figure \ref{fig:tt_alpha_3}, we have the edge lengths
 \begin{eqnarray*}
  C'_{i-1}E_i = \frac{c}{2}, & C'_{i-1}D_{i-1} = \dfrac{c_{\alpha}}{2}, & E_iW_i = \frac{s_{\alpha}}{2}, \\
  D_{i-1}V'_{i-1} = \frac{s}{2}, & C'_{i-1}V'_{i-1} = \dfrac{t}{2}, & C'_{i-1}W_i = \frac{|t_{\alpha}|}{2}. 
 \end{eqnarray*}

 By \cite[Page 454, 2.2.2 (vi)]{buser2010geometry}, we have
 \[
  \frac{\tanh C'_{i-1}V'_{i-1}}{\tanh C'_{i-1}D_{i-1}} = \frac{\tanh C'_{i-1} W_i}{\tanh C'_{i-1} E_i}, 
 \]
 namely, 
 \[
  \frac{\tanh \frac{|t_{\alpha}|}{2}}{\tanh \frac{c}{2}} = \frac{\tanh \frac{t}{2}}{\tanh \frac{c_{\alpha}}{2}}. 
 \]
 Note that $\gamma$ is a systole, then $c_{\alpha} \ge c$ and $|t_{\alpha}| \le t \le c \le c_{\alpha}$. 

 Notice that $s$ is well-defined even if on the $Y\in \mathcal{P}^{\pm}_{0,4} \subseteq \mathcal{T}_g/ \modo^{\pm}_g$ in which the choice of $\gamma$ is not unique, because for any choice of $\gamma$, the map induced by $\gamma$ maps $Y$ to a point in $\pi^{*}(F_0)$ and $\pi^{*}(F_0)$ is a fundamental domain of the covering $\Pi$. Recall $\Pi$ maps the Teichm\"uller curve $\pi^{*}(\mathcal{T}_{0,4}) \subseteq \mathcal{T}_g$ to its image in $\mathcal{T}_g/ \modo_g^{\pm}$. Therefore, for any choice of $\gamma$ on $Y$, the map induced by $\gamma$ is unique. 
 The proof is complete. 
\end{proof}

Now we are ready to prove Proposition \ref{prop:four_sys_cand}.
\begin{proof}[Proof of Proposition \ref{prop:four_sys_cand}]
 By Lemma \ref{lem:section}, for $X\in\mathcal{P}_g \cap \pi^{*}(F_0)$, $\gamma \in \mathrm{S}(X)$. What is left to prove is $\alpha, \beta, \gamma, \delta$ are the unique systole candidate in $L_{CD}, L_C, L_{CE}, L_{DE}$ respectively. 

 \begin{figure}[htbp]
  \centering
  \includegraphics{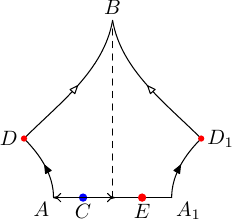}
  \caption{The pentagon $P$}
  \label{fig:half_pent_3}
 \end{figure}
 The multi-geodesic $\gamma$ is the unique systole candidate in $L_{CE}$ by Lemma \ref{lem:family_unique}. 
 To show this for the other three curve families, one may take a $4$-right-angled pentagonal fundamental domain $P$ of $O_3$ illustrated in Figure \ref{fig:half_pent_3}. In this pentagon, $B, C, E$ correspond to the singular points $B, C, E$ in $O_3$ respectively; $D$ and $D_1$ correspond to the singular point $D$ in $O_3$. Recall the orbifold $O_3$ is illustrated in Figure \ref{fig:22nn_222n_2}. For the lengths of the edges of this pentagon, one may see that 
 \[
  |AD| = |A_1D_1| = \frac{s}{2};
 \]
 \[
 |CE| = \frac{c}{2}, |AC| = \frac{t}{2}, |EA_1| = \frac{c-t}{2}.  
 \]

 Let $\eta \in L_{CD}$ be a systole candidate. The image of $\eta$ in $P$ consists of several geodesic segments as illustrated in Figure \ref{fig:half_pent_1}. Recall that $\gamma$ is a systole and its image in $P$ is the segment $A_1A$. Since $\eta$ is a systole candidate, any segment of $\eta$ in  $P$ is disjoint with the edge $A_1A$ except one segment joining $C$ by Lemma \ref{lem:intersect}. Notice that $P$ is axisymmetric with respect to the dashed line in Figure \ref{fig:half_pent_3} except the bottom edge $A_1A$. Then one may reflect some segments disjoint with $A_1A$ to obtain a piecewise geodesic segment, joining $C$ and $D$ (or $D_1$ ), with the same length as the image of $\eta$. 
 Therefore, the image of $\eta$ is longer than the segment $CD$ or $CD_1$. This implies  $\eta$ is longer than the simple closed geodesic lifts from $CD$ or $CD_1$ because the concerned curves join the same singular points in the orbifold $O_3$. 
 The last thing is comparing the length of $CD$ and $CD_1$. When $t \ne c$, $CD$ is shorter than $CD_1$. When $t =c$, $CD$ has the same length as $CD_1$. In this case, both $CD$ and $CD_1$ are not systoles by Lemma \ref{lem:family_unique}. Thus $CD$ lifts to the unique systole candidate in $L_{CD}$ when $X \in \pi^{*}(F_0)$. One can see that $CD$ lifts to $\alpha$ by assigning $t=0$. At this time, both $CD$ and $\alpha$ coincide with seams. Therefore, we have shown that $\alpha$ is the unique systole candidate in $L_{CD}$ when $X \in \pi^{*}(F_0)$. 
 
    \begin{figure}[htbp]
 \centering
 \begin{subfigure}[htbp]{.45\textwidth}
 \begin{center}
  \includegraphics{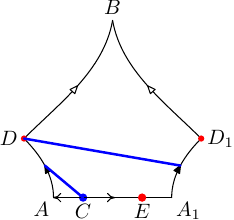}
 \end{center}
 \caption{$\eta$ of several segments}
 \label{fig:half_pent_1}
 \end{subfigure}
 \begin{subfigure}[htbp]{.45\textwidth}
 \begin{center}
  \includegraphics{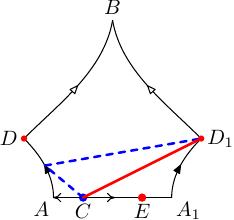}
 \end{center}
 \caption{Reflect some edges}
 \label{fig:half_pent_2}
 \end{subfigure}
 
 \caption{The image of $\eta$ in $P$}
 \label{fig:half_pent}
\end{figure}

    \begin{figure}[htbp]
 \centering
 \begin{subfigure}[htbp]{.45\textwidth}
 \begin{center}
  \includegraphics{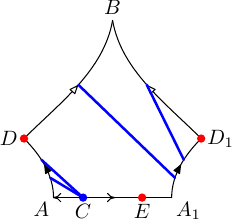}
 \end{center}
 \caption{$\eta$ of several segments}
 \label{fig:01_curve_2}
 \end{subfigure}
 \begin{subfigure}[htbp]{.45\textwidth}
 \begin{center}
  \includegraphics{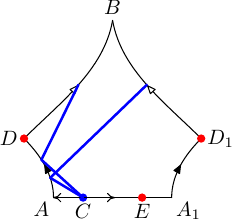}
 \end{center}
 \caption{Reflect some edges}
 \label{fig:01_curve_3}
 \end{subfigure}

 \begin{subfigure}[htbp]{.45\textwidth}
 \begin{center}
  \includegraphics{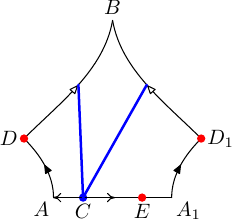}
 \end{center}
 \caption{$\beta$ in $P$}
 \label{fig:01_curve_4}
 \end{subfigure}
 
 \caption{The $L_C$ family in $P$}
 \label{fig:01_curve}
\end{figure}

    \begin{figure}[htbp]
 \centering
 \begin{subfigure}[htbp]{.45\textwidth}
 \begin{center}
  \includegraphics{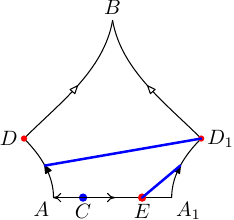}
 \end{center}
 \caption{$\eta$ of several segments}
 \label{fig:11_curve_1}
 \end{subfigure}
 \begin{subfigure}[htbp]{.45\textwidth}
 \begin{center}
  \includegraphics{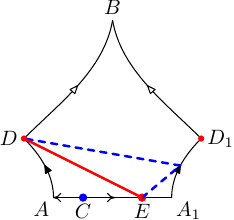}
 \end{center}
 \caption{Reflect some edges}
 \label{fig:11_curve_2}
 \end{subfigure}
 
 \caption{The curve family $L_{DE}$ in $P$}
 \label{fig:11_curve}
\end{figure}

The rest of the proof is to show $\beta, \delta$ are the unique systole candidates in $L_{DE}, L_C$ respectively. The proof is exactly the same as the case of $L_{CD}$, illustrated in Figure \ref{fig:01_curve} and \ref{fig:11_curve}. The proof is complete. 

\end{proof}

\subsection{The geodesic $\delta$}
The aim of this subsection is to prove Lemma \ref{lem:delta_not}, namely $\delta$ is not a systole when $X\in \mathcal{P}_g \cap \pi^{*}(F_0)$. To obtain this, some properties of $\ell_{\beta_i}$ ($i=1,2,..., g+1$) are needed. We first give these properties as a preparation. 

\begin{lemma}
 For the Fenchel-Nielsen coordinate $(c,t)$, when $0 \le t \le c$, we treat the length functions $\ell_{\beta_i} = \ell_{\beta_i}(c,t)$, $\ell_{\gamma_i} = \ell_{\gamma_i}(c,t)$ as functions of the variables $c, t$. 
 Consider the equation 
 \begin{equation}
  \ell_{\beta_i}(c,t) = \ell_{\gamma_i}(c,t). 
  \label{for:bc_eq}
 \end{equation}
 \begin{enumerate}
  \item If $t=c$, (\ref{for:bc_eq}) has a unique solution $c_1$. 
  \item If $t=0$, (\ref{for:bc_eq}) has no solution. 
  \item If $0<c<c_1$, (\ref{for:bc_eq}) has no solution. In particular, when $0<c<c_1$, $\ell_{\beta_i} > \ell_{\gamma_i}$. 
 \end{enumerate}
 \label{lem:bc_eqs}
\end{lemma}
\begin{proof}
 $(1)$ When $t=c$, by (\ref{for:beta}), we have
 \begin{eqnarray}
  \cosh \frac{\ell_{\beta_i}(c,c)}{2}
  &=& \cosh s \cosh^{2} \frac{c}{2} - \sinh^2 \frac{c}{2} \nonumber \\
  &=& (\cosh s -1) \sinh^2 \frac{c}{2} + \cosh s \nonumber \\
  &=& 2\sinh^2 \frac{s}{2} \sinh^2 \frac{c}{2} + \cosh s \nonumber\\
  &=& 2\cos^2 \frac{\pi}{g+1} + \cosh s \text{\,\,\,(by (\ref{for:seam_cuff})).} \label{for:cbeta}
 \end{eqnarray}
 By (\ref{for:seam_cuff}), $\ell_{\beta_i}(c,c)$ is strictly decreasing with respect to $c$, and $\lim\limits_{c\to 0^+}\ell_{\beta_i}(c,c) = + \infty$. On the other hand $\ell_{\gamma_i} (c,c) = 2c$. Thus $\ell_{\gamma_i}(c,c)$ is strictly increasing with respect to $c$, and $\ell_{\gamma_i}(c,c)=0$ when $c=0$. Therefore  (\ref{for:bc_eq}) has a unique solution. One may denote it as $c_1$. 

 $(2)$ When $t=0$, it follows directly from (\ref{for:t=0}). 

 $(3)$ For $0 \le t \le c < c_1$, suppose for contradiction that 
 \begin{equation}
 \ell_{\beta_i}(c,t) \le  \ell_{\gamma_i}(c,t) = 2c.
 \label{for:bct}
 \end{equation}
 By (\ref{for:pd_beta}), we have
 \begin{equation}
\ell_{\beta_i}(c,t) \ge \ell_{\beta_i}(c,c),
\label{for:betat}
 \end{equation}
 since $0 \le t \le c$. 
 We proved in $(1)$ that $\ell_{\beta_i}(c,c)$ is strictly decreasing with respect to  $c$. Thus when $c<c_1$, we have
\begin{equation}
 \ell_{\beta_i}(c,c) > \ell_{\beta_i}(c_1,c_1) = \ell_{\gamma_i}(c_1,c_1) = 2c_1 > 2c. 
 \label{for:bcc}
\end{equation}
Combining (\ref{for:bct}), (\ref{for:betat}), (\ref{for:bcc}), one may get a contradiction. The proof is complete.
\end{proof}
Consider a function
\begin{eqnarray*}
 F: [c_1, +\infty)\times (0,1] & \to \mathbb{R} \\
 (c,u) &\mapsto & \ell_{\beta_i}(c,cu) - \ell_{\gamma_i}(c,cu) = \ell_{\beta_i}(c,cu) - 2c. 
\end{eqnarray*}
We have the following two properties
\begin{lemma}
 There is a function 
 \[
  u_1:[c_1, +\infty) \to (0,1]
 \]
 such that $F(c, u_1(c)) = 0$. 
 \label{lem:u_implicit}
\end{lemma}
\begin{proof}
 The main ingredients to prove this lemma are Lemma \ref{lem:bc_eqs} and the implicit function theorem. 

 By Lemma \ref{lem:bc_eqs} $(1)$, we have
 \begin{equation*}
 F(c_1, 1) = 0. 
 \end{equation*}
 By (\ref{for:pd_beta}), we have
 \begin{equation*}
  \frac{\partial F}{\partial u} < 0
 \end{equation*}
 for $(c,u) \in [c_1, +\infty)\times (0,1]$. 
 By the implicit function theorem, there is a $u_1: [c_1, +\infty) \to (0,1]$ such that $F(c, u_1(c)) = 0$. 
The range $u_1( [c_1, +\infty)) \subseteq (0,1]$ 
because Lemma \ref{lem:bc_eqs} $(2)$ implies $u_1(c) >0$; the uniqueness part of Lemma \ref{lem:bc_eqs} $(1)$ implies $u_1(c) <1$ if $c>c_1$. 
The proof is complete. 
\end{proof}

For the function $u_1$ in Lemma \ref{lem:u_implicit}, we have
\begin{lemma}
 The preimage $u_1^{-1}( \frac{1}{2})$ consists of exactly one point. One may denote it as $c_{ \frac{1}{2}}$. In particular, when $g \ge 5$,  $c_{ \frac{1}{2}} < 2.318$; when $g = 3$, $c_{ \frac{1}{2}} < 1.925$. 
 \label{lem:u_level}
\end{lemma}

\begin{proof}
 By (\ref{for:gamma}) and (\ref{for:beta}), when $u = \frac{t}{c} = \frac{1}{2}$, $F(c, \frac{1}{2}) =0$ is equivalent to 
 \begin{eqnarray*}
  \cosh c &=& \cosh s \cosh \frac{c}{4}\cosh \frac{3c}{4} - \sinh \frac{c}{4}\sinh \frac{3c}{4} \\
   &=& \left(\cosh s -1\right)\cosh \frac{c}{4}\cosh \frac{3c}{4} + \cosh \frac{c}{4}\cosh \frac{3c}{4} - \sinh \frac{c}{4}\sinh \frac{3c}{4} \\
   &=& 2\sinh^2 \frac{s}{2} \cosh \frac{c}{4}\cosh \frac{3c}{4} + \cosh \frac{c}{2}\\
   &=&  \frac{\cos^2 \frac{\pi}{g+1}}{\sinh^2 \frac{c}{2}} \left( \cosh c + \cosh \frac{c}{2} \right) + \cosh \frac{c}{2}. 
 \end{eqnarray*}
 Let $C = \cosh \frac{c}{2}$ and $L = \cos^2 \frac{\pi}{g+1}$. Then we have
 \begin{equation}
  L = \frac{(C-1)^2 (2C+1)}{2C-1}. 
  \label{for:lc}
 \end{equation}
 Treat $L$ as a function of $C$. One may find that. 
 \[
  L(1) = 0; \lim\limits_{C\to +\infty} L(C) = +\infty. 
 \]
 Then for any $L\in(0,1)$, there is a $C > 1 $ satisfying the equation (\ref{for:lc}). 
 Moreover, we have
 \[
\frac{\mathrm{d} L}{\mathrm{d}C} = \frac{2(C-1)(4C^2 -2C +1 )}{(2C-1)^2}. 
 \]
 One may find that 
 \begin{equation}
 \frac{\mathrm{d} L}{\mathrm{d}C} >0,
 \label{for:dlc}
 \end{equation}
 when $C >1$. Thus, for a fixed $L\in(0,1)$ the number $C$  satisfying the equation (\ref{for:lc}) is unique. Let $c_{ \frac{1}{2}} \overset{\mathrm{def}}{=} 2 \arccosh (C)$ for $L = \cos^2 \frac{\pi}{g+1}$. We have proved the existence and uniqueness of $c_{ \frac{1}{2}}$. 

 The last thing is to estimate $c_{ \frac{1}{2}}$. When $g = 3$, $L = \cos^2 \frac{\pi}{4} = \frac{1}{2}$. Solving the equation (\ref{for:lc}), one may get 
 \[
  C = 1.5 \text{, namely } c_{ \frac{1}{2}} < 1.925.
 \]

 For $g \ge 5$, (\ref{for:dlc}) implies $\cosh \left(\frac{1}{2}c_{ \frac{1}{2}}  \right) $ is bounded from above by $C$ if taking $L=1$. At this time, one may get
 \[
  C < 1.75  ,
 \]
 which implies $c_{ \frac{1}{2}} < 2.318$. The proof is complete. 
\end{proof}

The function $u = u_1(c)$ is illustrated in Figure \ref{fig:uc}.
\begin{figure}[htbp]
 \centering
 \includegraphics{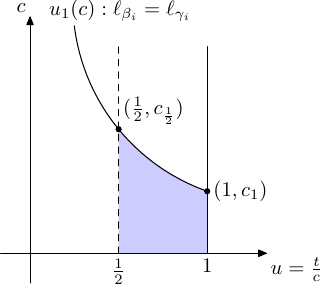}
 \caption{The function $u(c)$}
 \label{fig:uc}
\end{figure}

\begin{lemma}
 For $X \in \mathcal{P}_g \cap \pi^{*}(F_0)$, $\delta \notin \mathrm{S}(X)$. 
 \label{lem:delta_not}
\end{lemma}
\begin{proof}
 If $g$ is even, then for any $X \in \mathcal{P}_g \cap \pi^{*}(F_0)$, $\delta$ consists of one geodesic $\delta_1$. Thus $\delta_1$ intersects $\alpha_i$ and $\gamma_i$ twice for any $i= 1,2,...,g+1$, see \emph{e.g.} Figure \ref{fig:rotation_local_4}. 
 Suppose for contradiction that $\delta_1$ is a systole of some $X \in \mathcal{P}_g \cap \pi^{*}(F_0)$, then $\alpha_i$ and $\gamma_i$ are not systoles for $i=1,2,...,g+1$. On the other hand, $\delta_1$ and $\beta_1, \beta_2, ..., \beta_{g+1}$ do not fill $X$, because $\delta_1$ is a simple closed geodesic disjoint with $\{\beta_1, \beta_2, ..., \beta_{g+1}\}$. By Proposition \ref{prop:four_sys_cand}, $\mathrm{S}(X) \subseteq \left\{ \beta, \delta \right\} $, hence $X\notin \mathcal{P}_g$, we reach a contradiction. 

 For the case $g$ is odd, recall that $\delta$ consists of a pair of simple closed geodesics $\delta_1$ and $\delta_2$. Suppose for contradiction that $\delta$ is a systole of some $X \in \mathcal{P}_g \cap \pi^{*}(F_0)$. 
Recall since $X \in \pi^{*}(F_0)$, $\ell_{\gamma_i}(X) \le \ell_{\alpha_i}(X)$. On the other hand, $\left\{ \beta, \delta \right\} $ does not fill $X$. Thus for $X \in \mathcal{P}_g\cap \pi^{*}(F_0)$, $\left\{ \gamma, \delta  \right\} \subseteq \mathrm{S}(X)$. In other words, we have
\begin{equation}
  \ell_{\delta_j}(X) = \ell_{\gamma_i}(X) \le \ell_{\beta_i}(X) ;\, \ell_{\delta_j}(X) = \ell_{\gamma_i}(X) \le \ell_{\alpha_i}(X)
  \label{for:condition}
\end{equation}
 for $j=1,2$ and $i=1,2,...,g+1$. 

 For the Fenchel-Nielsen coordinate $(c,t)$, let $u \overset{\mathrm{def}}{=} \frac{t}{c}$. By (\ref{for:alpha}) and (\ref{for:delta_odd}), the condition $\ell_{\delta_j}(X)  \le \ell_{\alpha_i}(X)$ implies that
 \begin{equation}
  u \ge \frac{1}{2}. 
  \label{for:ad}
 \end{equation}
 On the other hand, by Lemma \ref{lem:bc_eqs} $(3)$, Lemma \ref{lem:u_implicit} and (\ref{for:pd_beta}), the condition $\ell_{\gamma_i}(X)  \le \ell_{\beta_i}(X)$ implies that
 \begin{equation}
  \text{either } c \le c_1 \text{ or }u \le u_1(c). 
  \label{for:bc}
 \end{equation}
 The points satisfying (\ref{for:ad}) and (\ref{for:bc}) are illustrated as the shaded area in Figure \ref{fig:uc}. Any point in this area satisfies 
 \begin{equation}
  c \le c_{ \frac{1}{2}}. 
  \label{for:c_up}
 \end{equation}
 By (\ref{for:seam_cuff}), one may get that
 \begin{equation}
  \sinh \frac{s}{2} \ge  {\cos\frac{\pi}{g+1}}\cdot \left( {\sinh \frac{c_{ \frac{1}{2}}}{2}} \right)^{-1}  \ge \begin{cases}
   \left({\sqrt{2}\sinh \frac{c_{ \frac{1}{2}}}{2}}  \right)^{-1}  , &\text{ if } g = 3 \\
   \left( {\frac{2}{\sqrt{3}}\sinh \frac{c_{ \frac{1}{2}}}{2}} \right)^{-1}  , &\text{ if } g  \ge 5
  \end{cases}. 
  \label{for:s_low}
 \end{equation}
 By (\ref{for:delta_odd}), one may get that
 \begin{equation}
  \ell_{\delta_i}(X) \ge (g+1)s \ge \begin{cases}
   4s, &\text{ if }g=3\\
   6s, & \text{ if }g \ge 5
  \end{cases}. 
  \label{for:d_low}
 \end{equation}
 Combining Lemma \ref{lem:u_level}, (\ref{for:c_up}), (\ref{for:s_low}), (\ref{for:d_low}), one may obtain that
 \begin{eqnarray*}
  \ell_{\delta_i}(X)&\ge& \begin{cases}
   4.77, &\text{ if }g = 3 \\
   6.85, &\text{ if }g \ge 5 \\
  \end{cases}; \\
  \ell_{\gamma_i}(X)&\le& \begin{cases}
   3.85, &\text{ if }g = 3 \\
   4.64, &\text{ if }g \ge 5 \\
  \end{cases}. 
 \end{eqnarray*}
 Thus $\ell_{\delta_i}(X) > \ell_{\gamma_i}(X)$, which contradicts to (\ref{for:condition}). The proof is complete. 
\end{proof}

\begin{proof}[Proof of Proposition \ref{prop:sys_cand}]
 The proof follows directly from Proposition \ref{prop:four_sys_cand} and Lemma \ref{lem:delta_not}. 
\end{proof}

\section{Thurston spine}
\label{sec:main}

In this section, we prove the main theorem. In the first subsection, Proposition \ref{prop:unique} is proved as a preparation. In the second subsection, the main theorem is proved. 

\subsection{Uniqueness of three systole point}
By the collar lemma (see \emph{e.g.} \cite[Theorem 4.1.1]{buser2010geometry}) and the Mumford compactness criterion (see \emph{e.g.} \cite{mumford1971remark}), there is a point $X_0 \in \pi^{*}(F_0)$ realizing the maximum of the systole function in $\pi^{*}(\mathcal{T}_{0,4})$, namely, $$\sys(X_0) \ge  \sys(X),\,\, \forall X \in \pi^{*}(F_0), $$
where $\sys(X)$ is the length of the systole of $X$. 
\begin{lemma}
 For $X\in \pi^{*}(\mathcal{T}_{0,4})$, if $\mathrm{S}(X) $ is a proper subset of $ \left\{ \alpha, \beta, \gamma \right\}$, then $X$ does not realize  the maximum of the systole function in $\pi^{*}(\mathcal{T}_{0,4})$. 

 \label{lem:max_abc}
\end{lemma}
\begin{proof}
 First consider the surface $X$, whose systoles consist of a single multi-geodesic among $\left\{ \alpha, \beta, \gamma \right\}$. Assume  $X\in \pi^{*}(\mathcal{T}_{0,4})$ satisfies $\mathrm{S}(X) = \{\eta\}$, for an $\eta \in \left\{ \alpha, \beta, \gamma \right\} $, where $\eta = \left\{ \eta_1, \eta_2, ...,\eta_{g+1} \right\} $. 
For any $X'$ in a sufficiently small neighborhood of $X$, one may have $$\mathrm{S}(X') = \{\eta\}.$$
 By (\ref{for:gamma}), (\ref{for:pd_alpha}), (\ref{for:pd_beta}), 
 there is a tangent vector $v \in T_{X} \pi^{*}(\mathcal{T}_{0,4})$ such that 
 \[
  \mathrm{d}\ell_{\eta_i}(v) >0, 
 \]
 for $i = 1,2,..., g+1$.  If $X'$ is on the flowline of the Weil-Petersson geodesic flow induced by $(X,v)$, then one may get that
 \[
  \ell_{\eta_i}(X) < \ell_{\eta_i}(X'). 
 \]
 Therefore, in a sufficiently small neighborhood of $X$, there is an $X'$ such that
 \[
  \sys(X) < \sys(X').
 \]
 Thus $X$ does not realize the maximum of the systole function on $\pi^{*}(\mathcal{T}_{0,4})$. 

 Consider the surface $X$ whose systoles consist of two multi-geodesics among $\left\{ \alpha, \beta, \gamma \right\}$. WLOG, one may assume $\mathrm{S}(X) = \left\{ \alpha, \beta \right\}$. For the surface $X$, if the projections of $\mathrm{d}\ell_{\alpha_i}$, $\mathrm{d}\ell_{\beta_i}$ on  $T^{*}_X \pi^{*}(\mathcal{T}_{0,4})$ are not vectors in opposite directions, then 
  there is a vector $v \in T_{X} \pi^{*}(\mathcal{T}_{0,4})$ such that 
 \[
  \mathrm{d}\ell_{\alpha_i}(v) >0, \mathrm{d}\ell_{\beta_i}(v) >0, 
 \]
 for $i = 1,2,..., g+1$. Then one may construct a surface $X'$ with larger systole than $X$ by $(X, v)$ as the single geodesic case. Therefore $X$ does not realize the maximum of the systole function on $\pi^{*}(\mathcal{T}_{0,4})$, if the projections of $\mathrm{d}\ell_{\alpha_i}$, $\mathrm{d}\ell_{\beta_i}$ on  $T^{*}_X \pi^{*}(\mathcal{T}_{0,4})$ are not vectors in opposite directions. 

If the projections of $\mathrm{d}\ell_{\alpha_i}$, $\mathrm{d}\ell_{\beta_i}$ on  $T^{*}_X \pi^{*}(\mathcal{T}_{0,4})$ are vectors in opposite directions, then 
  there is a vector $v \in T_{X} \pi^{*}(\mathcal{T}_{0,4})$ such that 
 \[
  \mathrm{d}\ell_{\alpha_i}(v) =0, \mathrm{d}\ell_{\beta_i}(v) =0, 
 \]
 Recall that 
 \[
  \mathrm{d}^2\ell_{\alpha_i}(v,v) >0, \mathrm{d}^2\ell_{\beta_i}(v,v) >0, 
 \]
 with respect to the Weil-Petersson metric by the strict convexity of length functions along Weil-Petersson geodesics (see \emph{e.g.} \cite{wolpert1987geodesic}). Then along the Weil-Petersson geodesic flow induced by $(X,v)$, there is a surface $X'$ such that
 \[
  \sys(X') > \sys(X), 
 \]
 when $X'$ is sufficiently close to $X$. Therefore $X$ does not realize the maximum of the systole function on $\pi^{*}(\mathcal{T}_{0,4})$, when $\mathrm{S}(X)$ consists of two of the three multi-geodesics. 
 The proof is complete. 

\end{proof}

\begin{proposition}
There is a unique point $X_0 \in \pi^{*}(F_0)$ such that $\mathrm{S}(X_0) = \{ \alpha, \beta, \gamma \}$. 
 \label{prop:unique}
\end{proposition}
\begin{proof}
 The proof of this proposition is an analog to \cite[Corollary 21]{schaller1999systoles}. 
 Assume $X_0$ is the point in $\pi^{*}(F_0)$, realizing the maximum of the systole function in $\pi^{*}(\mathcal{T}_{0,4})$. 
 By Proposition \ref{prop:sys_cand} and Lemma \ref{lem:max_abc}, $\mathrm{S}(X_0) = \left\{ \alpha, \beta, \gamma \right\} $. 
 Suppose for contradiction that there is another $X_0' \in \pi^{*}(F_0)$ whose systole is also $\{ \alpha, \beta, \gamma \}$. 
 Then there is a left earthquake flow in $\pi^{*}(\mathcal{T}_{0,4})$ flowing from $X_0'$ to $X_0$ by \cite[III.1.5.4.]{thurston2006earthquakes}. This path flows the set of geodesics $\{ \alpha, \beta, \gamma \}$ on $X_0'$ to the set of geodesics $\{ \alpha, \beta, \gamma \}$ on $X_0$. By the strict convexity of length functions along left earthquake path (see \emph{e.g.} \cite{kerckhoff1983nielsen}), for the vector $v \in T_{X_0}\pi^{*}(\mathcal{T}_{0,4})$ tangent to the flow, one may have
 \[
  \mathrm{d}\ell_{\alpha_i}(v) \ge 0, \mathrm{d}\ell_{\beta_i}(v) \ge 0, \mathrm{d}\ell_{\gamma_i}(v) \ge 0, 
 \]
 for $i =1,2,..., g+1$. As $\{ \alpha, \beta, \gamma \}$ fills the surface, at least one of the three numbers is $>0$. 
 Then by flowing along an earthquake path induced by $(X_0,v)$, one may get a surface $X'$ near $X_0$ such that 
 \[
  \sys(X_0) \le \sys(X'). 
 \]
 If $\sys(X_0) < \sys(X')$, then  $X_0$ does not realize the maximum of the systole function in $\pi^{*}(\mathcal{T}_{0,4})$. A contradiction is reached. 
If $\sys(X_0) = \sys(X')$, then
 $\mathrm{S}(X')$ is a proper subset of $\{ \alpha, \beta, \gamma \}$. 
 By Lemma \ref{lem:max_abc}, $X'$ is not a point realizing the maximum of the systole function in $\pi^{*}(F_0)$, hence $X_0$ is also not. A contradiction is reached, and the proof is complete. 
\end{proof}
\subsection{The main theorem}
Assume the $X_0\in \pi^{*}(F_0)$ realizing the maximum of the systole function in $\pi^{*}(\mathcal{T}_{0,4})$ has Fenchel-Nielsen coordinate $(c_M, t_M)$ and $u_M = \frac{t_M}{c_M}$. 
Consider the Fenchel-Nielsen coordinate $(c,t)$ and $u= \frac{t}{c}$. Lemma \ref{lem:u_implicit} provides a curve in $(c,u) \in (0,+\infty)\times (0,1]$, parametrized by the function $u_1:[c_1, +\infty)\to (0,1]$, consisting of the points satisfying that
\[
 \ell_{\beta_i}(X) = \ell_{\gamma_i}(X)
\]
for $i =1,2,...,g+1$, when $(c,u)\in (0,+\infty)\times (0,1]$. 

Recall the Teichm\"uller geodesic $L_{-1,1}$ consists of the points $X$ such that
\[
 \ell_{\alpha_i}(X) = \ell_{\gamma_i}(X)
\]
for $i =1,2,...,g+1$. 
Similar to Lemma \ref{lem:u_implicit}, one can show that
\begin{lemma}
 There is a function $u_0:[c_0, +\infty)\to [0,1)$ such that its graph $(c, u_0(c))$ is the geodesic $L_{-1,1}$ in $\pi^{*}(F_0)$. 
 \label{lem:u0_implicit}
\end{lemma}
\begin{proof}
 By (\ref{for:cc_alpha_t0}), there is a unique point $(c_0,0)$ on the geodesic $\left\{ u=0 \right\} $ such that $\ell_{\gamma_i} = \ell_{\alpha_i}$. 
 Similar to Lemma \ref{lem:bc_eqs} $(3)$, by (\ref{for:pd_alpha}), for $(c,u)\in (0,+\infty) \times (0,1)$, if $\ell_{\alpha_i}(c,cu) = \ell_{\gamma_i}(c,cu)$, then  $c> c_0$. 
 Consider the function
 \[
  G(c,u) \overset{\mathrm{\mathrm{def}}}{=} \ell_{\alpha_i}(c,cu) - \ell_{\gamma_i}(c,cu) = \ell_{\alpha_i}(c,cu) - 2c
 \]
 defined on $(c,u)\in (0,+\infty) \times [0,1]$. 
 By (\ref{for:pd_alpha}), one may get that
 \[
  \frac{\partial G}{\partial u} (c,u) > 0 
 \]
 for $(c,u)\in (0,+\infty) \times [0,1]$. 
 Then by the implicit function theorem, there is a function 
 \[
  u_0: [c_0, +\infty) \to [0,1)
 \]
 such that 
 \[
  G(c, u_0(c)) = 0. 
 \]
 The range $u_0([c_0,+\infty))$ follows from Lemma \ref{lem:l_pm1}. The proof is complete.

\end{proof}

Now we are ready to state the main theorem. 
Let 
\begin{eqnarray*}
 \Gamma(u_1([c_1,c_{M}])) &\overset{\mathrm{def}}{=}& \left\{ (c, u )\,|\, c\in [c_1,c_{M}], u = u_1(c) \right\}; \\
 \Gamma(u_0([c_0,c_{M}])) &\overset{\mathrm{def}}{=}& \left\{ (c, u )\,|\, c\in [c_0,c_{M}], u= u_0(c)\right\}.
\end{eqnarray*}
\begin{theorem}
 \[
  \pi^{*}(F_0) \cap \mathcal{P}_g = \Gamma(u_1([c_1, c_M]))\cup \Gamma(u_0([c_0, c_M])). 
 \]
 \label{thm:main}
\end{theorem}
\begin{proof}
 By (\ref{for:t=0}), the point $(c_0,0) \in \mathcal{P}_g$. More precisely, we have
 \[
  \ell_{\alpha_i}(c_0,0) = \ell_{\gamma_i}(c_0,0) < \ell_{\beta_i}(c_0,0)
 \]
 for $i =1,2,..., g+1$. 
 For $c \in [c_0,c_M)$ the inequality 
 \[
  \ell_{\alpha_i}(c, u_0(c)\cdot c) = \ell_{\gamma_i}(c,u_0(c)\cdot c) < \ell_{\beta_i}(c,u_0(c)\cdot c)
 \]
 for $i =1,2,..., g+1$ 
 always holds by Proposition \ref{prop:unique}. Hence $(c,u_0(c)) \in \mathcal{P}_g$. Notice that $(c,u_0(c))\in L_{-1,1}\cap \pi^{*}(F_0) \subseteq \pi^{*}(F_0)$. 

 On the other hand, by (\ref{for:t=c}), the point $(c_1,1) \in \mathcal{P}_g$. More precisely, we have
 \[
  \ell_{\beta_i}(c_1,1\cdot c_1) = \ell_{\gamma_i}(c_1,1\cdot c_1) < \ell_{\alpha_i}(c_1,1\cdot c_1)
 \]
 for $i =1,2,..., g+1$. 
 For $c \in [c_1,c_M)$ the inequality 
 \[
  \ell_{\beta_i}(c,u_1(c)\cdot c) = \ell_{\gamma_i}(c,u_1(c)\cdot c) < \ell_{\alpha_i}(c,u_1(c)\cdot c)
 \]
 for $i =1,2,..., g+1$
 always holds by Proposition \ref{prop:unique}. Hence $(c,u_1(c)) \in \mathcal{P}_g$. By Lemma \ref{lem:u_implicit} and Proposition \ref{prop:unique}, $(c,u_1(c)) \in \pi^{*}(F_0)$. 
 Therefore we have proved $\Gamma(u_1([c_1, c_M]))\cup \Gamma(u_0([c_0, c_M])) \subseteq  \pi^{*}(F_0) \cap \mathcal{P}_g$. 

 The last thing is to show $\Gamma(u_1([c_1, c_M]))\cup \Gamma(u_0([c_0, c_M])) =  \pi^{*}(F_0) \cap \mathcal{P}_g$. Recall the domain $\pi^{*}(F_0)$ is bounded by three geodesics $L_0 = \left\{ u=0 \right\} $, $L_1 = \left\{ u=1 \right\}$ and $L_{-1,1} = \left\{ \ell_{\alpha_i} = \ell_{\gamma_i} \right\} $. By Lemma \ref{lem:bc_eqs} and Lemma \ref{lem:u0_implicit}, one may have have
 \[
  L_0 \cap \mathcal{P}_g = \left\{ (c_0,0) \right\} ; L_1 \cap \mathcal{P}_g = \left\{ (c_1,1\cdot c_1) \right\}. 
 \]
By Proposition \ref{prop:unique}, for $c > c_M$, one may have
 \[
  \ell_{\alpha_i}(c,u_0(c)\cdot c) = \ell_{\gamma_i}(c,u_0(c)\cdot c) > \ell_{\beta_i}(c,u_0(c)\cdot c)
 \]
 for $i =1,2,..., g+1$. Therefore, one may have
 \[
  L_{-1,1} \cap \mathcal{P}_g \cap \pi^{*}(F_0) = \Gamma(u_0([c_0,c_M])). 
 \]
 For $(c,t)$ and $u = \frac{t}{c}$ in the interior of $\pi^{*}(F_0)$, one may have
 \[
  \ell_{\alpha_i}(c,u\cdot c) > \ell_{\gamma_i}(c,u\cdot c). 
 \]
 For any fixed $c>0$, there is at most one $u\in (0,1]$ satisfying 
 \[
  \ell_{\beta_i}(c,u\cdot c) = \ell_{\gamma_i}(c,u\cdot c). 
 \]
 This point is the point $(c,u_1(c))$ when  $c \ge c_1$. This point is in $\pi^{*}(F_0)$ if and only if $c \le c_{M}$ by Proposition \ref{prop:unique}. 
 Therefore $\Gamma(u_1([c_1, c_M]))\cup \Gamma(u_0([c_0, c_M])) =  \pi^{*}(F_0) \cap \mathcal{P}_g$. The proof is complete. 

\end{proof}

By the reflection $r_0$, $r_1$, one may get $\mathcal{P}_g \cap \pi^{*}(F)$, as illustrated in Figure \ref{fig:fund_dom_2}. 
By the action of $\pmodo_{0,4} = \left< D_{\alpha}, D_{\gamma} \right>$, one may get the Thurston spine in $\pi^{*}(\mathcal{T}_{0,4})$ and $\pi^{*}(\mathcal{M}_{0,4})$, as illustrated in Figure \ref{fig:spine}.

One may see that
\begin{corollary}

	For the Teichm\"uller curve $\pi^{*}(\mathcal{T}_{0,4})$ and the Thurston spine $\mathcal{P}_g$, the intersection  $\pi^{*}(\mathcal{T}_{0,4}) \cap \mathcal{P}_g$ is an equivariant deformation retract of $\pi^{*}(\mathcal{T}_{0,4})$. 
 \label{cor:conj}
\end{corollary}

In Figure \ref{fig:spine}, one circle realizes $D_{\alpha}$ and the other circle realizes $D_{\gamma}$. Then we have that
\begin{corollary}
	Any element in $\left<D_{\alpha}, D_{\gamma} \right> \subseteq \modo_g$ can be realized as an essential loop in $q(\pi^{*}(\mathcal{T}_{0,4})\cap \mathcal{P}_g) \subseteq \mathcal{M}_g$. 
	\label{cor:loop}
\end{corollary}
The group $\left<D_{\alpha}, D_{\gamma} \right> \subseteq \modo_g$ contains both reducible and pseudo-Anosov elements, see \emph{e.g.} \cite{thurston1988geometry}.

 \bibliographystyle{alpha}
 \bibliography{spine_example}
\end{document}